\documentclass[1p]{elsarticle}
\usepackage{hyperref}
\usepackage{amssymb}
\usepackage{latexsym}
\usepackage{amscd}
\usepackage{amsthm}
\usepackage{amsfonts}
\usepackage{CJK}
\usepackage{CJKulem}
\usepackage{color}
\usepackage{colortbl}
\usepackage{fancyhdr}
\usepackage{indentfirst}
\usepackage{lastpage}
\usepackage{latexsym}
\usepackage{listings}
\usepackage{multirow}
\usepackage{mathrsfs}
\usepackage{placeins}
\usepackage{titlesec}
\usepackage{ulem}
\usepackage{graphicx} 
\usepackage{amsmath} 
\usepackage{amssymb} 
\usepackage{yfonts}
\usepackage{xypic}
\usepackage{bm}

\newdimen\AAdi%
\newbox\AAbo%
\def\AAk#1#2{\s_etbox\AAbo=\hbox{#2}\AAdi=\wd\AAbo\kern#1\AAdi{}}%
\def\AAr#1#2#3{\s_etbox\AAbo=\hbox{#2}\AAdi=\ht\AAbo\raise#1\AAdi\hbox{#3}}%
\font\tenmsb=msbm10 at 12pt \font\sevenmsb=msbm7 at 8pt
\font\fivemsb=msbm5 at 6pt
\newfam\msbfam
\textfont\msbfam=\tenmsb \scriptfont\msbfam=\sevenmsb
\scriptscriptfont\msbfam=\fivemsb
\def\Bbb#1{{\tenmsb\fam\msbfam#1}}
\newtheorem{thm}{Theorem}[section]
\newtheorem{main-thm}{Main Theorem}
\newtheorem{lem}[thm]{Lemma}
\newtheorem{cor}[thm]{Corollary}
\newtheorem{pro}[thm]{Proposition}

\newtheorem{con}[thm]{Conjecture}
\theoremstyle{remark}\newtheorem{rem}{Remark}[section]

\newcommand{\Section}[2]{\setcounter{equation}{0}
	\allowdisplaybreaks
	\section[#1]{#2}}

\def\n{\nabla}

\def\f#1#2{\frac{#1}{#2}}

\def\a{\alpha}
\def\be{\beta}

\def\de{\delta}
\def\De{\Delta}

\def\ep{\varepsilon}

\def\g{\gamma}

\def\la{\lambda}

\def\Om{\Omega}
\def\th{\theta}

\def\R{\Bbb{R}}

\def\lan{\langle}
\def\ran{\rangle}
\def\ra{\rightarrow}

\def\ol{\overline}

\begin{document}
	\title
	{An optimal pinching theorem on compact minimal submanifolds in the Euclidean spheres via eigenvalues of fundamental matrices}
	
	\author[f]{Huimin Liu}
	\ead{hmliu20@fudan.edu.cn}
	\author[f,m]{Ling Yang}
	\ead{yanglingfd@fudan.edu.cn}
	
	\address[f]{School of Mathematical Sciences, Fudan University, Shanghai, 200433, China}
	\address[m]{Shanghai Center for Mathematical Sciences, Shanghai, 200438, China}


	\begin{abstract}
    Based on an inequality on the upper bound of the sum of squared norms of Lie bracket
    of symmetric matrices, we establish a rigidity theorem for compact minimal submanifolds
    in the Euclidean spheres via eigenvalues of fundamental matrices, which are the critical values
    of the squared norms of the second fundamental form on all normal directions. This conclusion is
    optimal for all dimensions without the restriction of the codimension, giving a new characterization for
    generalized Clifford tori and Veronese manifolds.


	\end{abstract}

	
	\maketitle
	
	\tableofcontents
	
	\renewcommand{\proofname}{\it Proof.}
	
	
	\Section{Introduction}{Introduction}

    In 1968, J. Simons \cite{S} proved a well-known rigidity theorem on minimal submanifolds in Euclidean spheres as follows:
    \begin{thm}\label{Si1}
    Let $M$ be an $n$-dimensional compact minimal submanifold
    in $S^{n+m}$, denote by $|B|^2$ be the square norm of the second fundamental form $B$ of $M$, then
    \begin{equation}\label{Si}
    \int_M |B|^2\left[\left(2-\f{1}{m}\right)|B|^2-n\right]dM\geq 0.
    \end{equation}
    As a corollary, the pinching condition $0\leq |B|^2\leq \f{n}{2-\f{1}{m}}$ forces $|B|^2\equiv 0$ or $|B|^2\equiv \f{n}{2-\f{1}{m}}$.
    \end{thm}

    As shown by Chern-do Carmo-Kobayashi \cite{C-D-K} and B. Lawson \cite{L}, $|B|^2\equiv \f{n}{2-\f{1}{m}}$ implies $M$ is a generalized Clifford torus or a Veronese surface. 
    Namely,
    Simons' theorem describes the first gap of $|B|^2$, which is optimal for the hypersurface cases (i.e. $m=1$) or the 2-dimensional cases.
    This conclusion is an important application of Bochner's technique in the theory of minimal submanifolds.
    For each $x\in M$, let $\{e_1,\cdots,e_n\}$ and $\{\nu_1,\cdots,\nu_m\}$ be orthonormal basis of the tangent space
   and the normal space at $x$, respectively,
   \begin{equation}
   A^\a:=\Big(\lan B_{e_ie_j},\nu_\a\ran\Big)\qquad \forall \a=1,\cdots,m
   \end{equation}
   be the matrix of the second fundament form w.r.t. each normal direction, then
   \begin{equation}\label{Si2}
   \f{1}{2}\De |B|^2=|\n B|^2+n|B|^2-\sum_{\a,\be}\|[A^\a,A^\be]\|^2-\sum_{\a,\be}\big(\text{tr}(A^\a A^\be)\big)^2,
   \end{equation}
   where $\De$ is the Laplace-Beltrami operator on $M$, $[\cdot,\cdot]$ is the Lie bracket and $\|\cdot\|$ denotes the Hilbert-Schmidt norm of matrices (see (\ref{HS})).
   Thereby, Simons' inequality (\ref{Si}) directly follows from (\ref{Si2}) and the following estimate
   \begin{equation}
   \|[A,B]\|^2\leq 2\|A\|^2\|B\|^2
   \end{equation}
   with $A,B$ being both $(n\times n)$-matrices. Note that the terms $\|[A^\a,A^\be]\|^2$
   have a close relationship to the curvature of the normal bundle of $M$, which automatically vanish
   whenever $M$ is a hypersurface. Therefore, how to give an estimate for these terms is one of the key points to studying the rigidity
   of compact minimal submanifolds of higher codimension. This work has been done by Li-Li \cite{L-L} and Chen-Xu \cite{C-X} independently.
   Based on
   \begin{equation}
   \sum_{\a,\be}\|[A^\a,A^\be]\|^2+\sum_{\a,\be}\big(\text{tr}(A^\a A^\be)\big)^2\leq \f{3}{2}\Big(\sum_\a \|A^\a\|^2\Big)^2,
   \end{equation}
   they derived a rigidity theorem whose condition is weaker than Theorem \ref{Si1}:
    \begin{thm}\label{Li}
   Let $M$ be an $n$-dimensional compact minimal submanifold
    in $S^{n+m}$ with $m\geq 2$. If $|B|^2\leq \f{2n}{3}$, then $M$ is a totally geodesic subsphere $(|B|^2\equiv 0)$,
     or the Veronese surface $($here $n=2$ and $|B|^2\equiv \f{4}{3}$$)$.
   \end{thm}

   Let
    \begin{equation}
   (S_{\a\be}):=\big(\text{tr}(A^\a A^\be)\big)=\left(\sum_{i,j}\lan B_{e_ie_j},\nu_\a\ran \lan B_{e_ie_j},\nu_\be\ran\right)
   \end{equation}
be the {\it fundamental matrix} at $x$. Z. Q. Lu \cite{Lu} studied the rigidity problem via eigenvalues of $(S_{\a\be})$ and established
   the following pinching theorem:
   \begin{thm}\label{Lu1}
Let $M$ be an $n$-dimensional compact minimal submanifold in $S^{n+m}$ and $\la_2$ be the second large eigenvalue of the fundamental matrix
at each point. If $|B|^2+\la_2\leq n$, then $M$ is a totally geodesic subsphere $(|B|^2+\la_2\equiv 0)$,
a generalized Clifford torus $(|B|^2\equiv n$ and $\la_2\equiv 0)$ or the Veronese surface $(|B|^2+\la_2\equiv 2)$.
\end{thm}
Observing that $|B|^2=\text{tr}(S_{\a\be})=\sum\limits_{\a} \la_\a$, $|B|^2\leq \f{2n}{3}$ implies $|B|^2+\la_2\leq n$ and hence
Theorem \ref{Lu1} is an improvement of Theorem \ref{Li}. However, for the cases of $n\geq 3$, it is still unknown whether there exists
a pinching condition forcing $M$ to be a minimal submanifold with nonflat normal bundle.

In this paper, we shall do essentially better, establishing the following pinching theorem on eigenvalues of fundamental matrices:
\begin{thm}\label{main-thm1}
Let $M$ be an $n$-dimensional compact minimal submanifold in $S^{n+m}$,
$\la_1\geq \cdots \geq \la_m$ be the eigenvalues of the fundamental matrix at each point, if
\begin{equation}\label{fun-eig}
\sum_{\a=1}^{\min\{n,m\}}\la_\a+\la_2\leq n,
\end{equation}
then $M$ is totally geodesic $\left(\sum\limits_{\a=1}^{\min\{n,m\}}\la_\a+\la_2 \equiv 0\right)$ or $\sum\limits_{\a=1}^{\min\{n,m\}}\la_\a+\la_2 \equiv n$.
\end{thm}
This condition is optimal for all dimensions without the restriction of codimension, since both generalized Clifford tori and Veronese manifolds are compact minimal submanifolds
in Euclidean spheres, satisfying $\sum\limits_{\a=1}^{\min\{n,m\}}\la_\a+\la_2 \equiv n$ (see Theorem \ref{Clifford-Veronese}). Moreover,
we establish a new characterization for such two submanifolds of dimension 2 and 3 as follows.
\begin{thm}\label{main-thm}
Let $M$ be an $n$-dimensional compact minimal submanifold in $S^{n+m}$, satisfying $\sum\limits_{\a=1}^{\min\{n,m\}}\la_\a+\la_2 \equiv n$, then
\begin{itemize}
\item If $n=2$, then $M$ is a Clifford torus $(\la_1\equiv 2$, $\la_2\equiv 0)$ or a Veronese surface $(\la_1\equiv\la_2\equiv\f{2}{3},\la_3\equiv 0)$;
\item If $n=3$, then $M$ is a generalized Clifford torus $(\la_1\equiv 3$, $\la_2\equiv 0)$ or a Veronese $3$-manifold $(\la_1\equiv \cdots\equiv \la_5\equiv \f{3}{4},\la_6\equiv 0)$.
\end{itemize}
\end{thm}

\begin{rem}
 Let $i:M\ra S^{n+m}$ be the inclusion map and $\phi:S^{n+m}\ra S^{n+m+k}$ be the totally geodesic embedding,
then $\phi\circ i$ gives a minimal immersion from $M$ into $S^{n+m+k}$, keeping $\la_1,\cdots,\la_m$ invariant at each $x\in M$.
(Here $\la_{m+1}\equiv \cdots \equiv \la_{m+k}\equiv 0$.)
This observation means, we can assume $m$ is sufficiently large without loss of generality. Noting that the fundamental matrix $S$ at $x\in M$
is diagonalizable, we can assume
\begin{equation}
\lan A^\a,A^\be\ran:=\text{tr}(A^\a A^\be)=\la_\a\de_{\a\be}.
\end{equation}
Therefore, the number
of nonzero eigenvalues of $S$ cannot exceed $\f{(n-1)(n+2)}{2}$, the dimension of the vector space $\text{Sym}_n^0$ of all traceless symmetric $(n\times n)$-matrices.
This shows, for the case of dimension 2, Theorem \ref{main-thm} and Theorem \ref{Lu1} are equivalent.
\end{rem}

In fact, we claim Theorem \ref{main-thm} can be generalized to all dimensions. The details of the proof shall be given in the subsequent papers
of the authors of the present paper.

Z. Q. Lu's work \cite{Lu} is based on the following inequality
\begin{equation}\label{Lu0}
\sum_{\a=2}^m\left\|[A^1,A^\a]\right\|^2\leq \|A^1\|^2\left(\|A^2\|^2+\sum_{\a=2}^{m}\|A^\a\|^2\right),
\end{equation}
which shall be improved in this paper, laying the algebraic foundation of establishing the pinching theorem. Noting that
\begin{equation}
\left\|[A^1,A^\a]\right\|^2=\lan (\operatorname{ad} A^1)^2 A^\a,A^\a\ran
\end{equation}
with $(\operatorname{ad} A^1)^2$ a self-adjoint nonnegative definite transformation on $\text{Sym}_n^0$,
the LHS of (\ref{Lu0}) can be estimated via the eigenvalues of $(\operatorname{ad} A^1)^2$, i.e. the squared differences of eigenvalues of $A^1$. From this viewpoint, a more refined inequality
\begin{equation}\label{main-ineq}
\sum_{\a=2}^m\left\|[A^1,A^\a]\right\|^2\leq \|A^1\|^2\left(\|A^2\|^2+\sum_{\a=2}^{\min\{m,n\}}\|A^\a\|^2\right)
\end{equation}
shall be set up on the basis of Lemma \ref{eig} and Lemma \ref{sq-diff}. Afterwards, following the arguments
of \cite{Lu}, we calculate the Laplacian of $g_p:=\text{tr}(S^p)^{\f{1}{p}}$ and then derive a Simons type
integral inequality (see Proposition \ref{p1}), which immediately yields the conclusions of Theorem \ref{main-thm1}.
It should be pointed out that the classification of the submanfolds when the equality in (\ref{fun-eig}) holds
is a difficult task, so we only deal with the situation of $n=3$ in this paper. In this case, such manifold $M$ can be decomposed into
3 domains $\Om_1,\Om_2,\Om_3$, where $\Om_1:=\{x\in M:\la_3>\la_4\}$, $\Om_2:=\{x\in M:\la_2>\la_3\}$
and $\Om_3:=\{x\in M: \la_2=\la_3=\la_4\}$. For each $x\in \Om_1$ or $\Om_2$, since the equality of (\ref{main-ineq})
holds pointwisely, we can find an orthonormal tangent frame field and an orthonormal normal frame field on
a neighborhood of $x$, such that $A^1$ can be written as a specific diagonal matrix and $A^\a$ are all eigenvectors of $(\operatorname{ad} A^1)^2$.
In conjunction with the additional information on $\n B$ (see \ref{c2}), we can use the Codazzi equations to establish several relationships between
the curvature tensors of the tangent bundle and the normal bundle, which cause contradictions to the conclusions deriving from the Gauss equations
and the Ricci equations. Therefore both $\Om_1$ and $\Om_2$ are empty sets, which enable us to estimate a lower bound of $\De |B|^2$ pointwisely
and then clarify the second fundamental form of $M$. In virtue of the works of Chern-do Carmo-Kabayashi \cite{C-D-K} and Itoh-Ogiue \cite{I-O}, $M$ is shown to be
a generalized Clifford torus or a Veronese 3-manifold.

Throughout history, a conclusion for the first gap of a geometric quantity always becomes a beginning of exploring the rigidity for minimal submanifolds in Euclidean spheres
from the corresponding viewpoint.  Based on Simon's theorem (Theorem \ref{Si1}), S. S. Chern \cite{C} raised a well-known conjecture that
$|B|^2$ takes values in a discrete set for each compact minimal submanifold in Euclidean spheres satisfying $|B|^2\equiv \text{const}$.
 For the hypersurface cases, Peng-Terng \cite{P-T,P-T2} made the first
 effort to the Chern conjecture and confirmed the second gap of $|B|^2$. This beautiful work attracts a lot of successive studies,
 see \cite{Ch,Y-C1,Y-C2,Y-C3,S-Y,W-X,Z,D-X,X-X,L-X-X,G-T-Y-Z,D-K,Chen}. Recently, Tan-Tang-Xie-Yan \cite{T-T-X-Y} got the best result in this subject up to now,
proving local finiteness of the values of $|B|^2$ for compact embedded minimal hypersurfaces in Euclidean spheres satisfying $|B|^2\equiv \text{const}$
and $f_3\equiv \text{const}$. In the higher codimensional situation, Chern's conjecture was shown to be true for 2-dimensional cases by R. Bryant \cite{B}.
However, a countable family of compact embedded minimal submanifolds constructed by Firester-Tsiamis \cite{F-T} disproved this conjecture for
all $n\geq 3$ with $m\geq 4$ and for even $n\geq 4$ with $m\geq 3$. On the other hand, Z. Q. Lu \cite{Lu} pointed out $|B|^2+\la_2$ might be the right object to study rigidity for minimal submanifolds in spheres. He conjectured that there exists a positive constant $\ep$, such that $|B|^2+\la_2$
cannot takes values in $(n,n+\ep)$ for all $n$-dimensional compact minimal submanifolds in $S^{n+m}$ satisfying $|B|^2+\la_2\equiv \text{const}$.
Recently, this conjecture was proved by  for $n=m=2$ by Ge-Li-Zhang \cite{G-L-Z}, and disproved by Li-Zhao \cite{L-Z} for $n=2$ and $m\geq 3$.

Inspired by the main theorems of the present paper, we raise the following conjecture:
\begin{con}
Let $M$ be an $n$-dimensional compact minimal submanifold in $S^{n+m}$ that is homeomorphic (or diffeomorphic) to $S^n$,
$\la_1\geq \cdots \geq \la_m$ be the eigenvalues of the fundamental matrix at each point, then
there exists a constant $\ep$ depending on $n$, such that
\begin{equation}
n \leq \sum_{\a=1}^{\min\{n,m\}}\la_\a+\la_2\leq n+\ep
\end{equation}
forces $M$ to be a Veronese manifold.

\end{con}

	\bigskip\bigskip

\Section{An improvement of Lu's inequality}{An improvement of Lu's inequality}\label{inequality}
\subsection{Lu's inequality}

For two real $(n\times n)$-matrices $A=(a_{ij})$ and $B=(b_{ij})$, let
\begin{equation}
\lan A,B\ran:=\sum_{i,j=1}^n a_{ij}b_{ij}=\text{tr }(AB^T)
\end{equation}
be the inner product of $A$ and $B$, where
 $(\cdot)^T$ denotes the transpose of a matrix. This induces the Hilbert-Schmidt norm:
\begin{equation}\label{HS}
\|A\|=\sqrt{\lan A,A\ran}=\sqrt{\sum_{i,j=1}^n a_{ij}^2}.
\end{equation}

Z. Q. Lu  established the following matrix inequality as the main algebraic tool of his work \cite{Lu} on the rigidity problem for
higher codimensional minimal submanifolds in Euclidean spheres.
\begin{pro}\label{Lu}
Let $A_1,A_2,\cdots,A_m$ be $(n\times n)$-symmetric matrices, such that
\begin{itemize}
\item $\|A_1\|=1$, $\text{tr}(A_1)=0$;
\item $\lan A_\a,A_\be\ran=0$ for each $2\leq \a< \be\leq m$;
\item $\|A_2\|\geq \cdots\geq \|A_m\|$.
\end{itemize}
Then
\begin{equation}
\sum_{\a=2}^m\left\|[A_1,A_\a]\right\|^2\leq \|A_2\|^2+\sum_{\a=2}^{m}\|A_\a\|^2
\end{equation}
and the equality holds if and only if, after an orthonormal base change and up to a sign, we have
\begin{equation}
A_1=\la\left(\begin{matrix} k & & \\ & -I_k &\\ & & O\end{matrix}\right),
\end{equation}
$A_\a$ $(2\leq \a\leq k+1)$ is $\mu$ times the matrix whose only nonzero entries are $1$ at the $(1,\a)$ and $(\a,1)$ places,
i.e. $A_\a=\mu(E_{1\a}+E_{\a 1})$ and $A_{k+2}=\cdots=A_m=0$. Here $1\leq k\leq m-1$, $\la=\f{1}{\sqrt{k(k+1)}}$ and $\mu$ is a constant.  

\end{pro}

\begin{rem}
The above conclusion is just Lemma 2.4 of \cite{W}, which is the revised version of Lemma 2 of \cite{Lu}.
Here the author found there are more cases when the Lu's equality holds and gave another proof by using Lagrange Muliplier method.
\end{rem}

We shall make an improvement on Lu's inequality, based on the following 2 lemmas.

\subsection{A lemma on eigenvalues}

The first one is an estimate in terms of eigenvalues of a given non-negative definite
self-adjoint operator, which has its own interest.
\begin{lem}\label{eig}
Let $\Phi$ be a non-negative definite self-adjoint transformation on an $N$-dimensional inner product space $V$,
and $v_1,\cdots,v_m$ be vectors in $V$ satisfying
\begin{itemize}
\item $\lan v_\a,v_\be\ran=0$ whenever $\a\neq \be$;
\item $|v_1|\geq |v_2|\geq \cdots \geq |v_m|$.
\end{itemize}
Then
\begin{equation}
\sum_{\a=1}^m \lan \Phi(v_\a),v_\a\ran\leq \sum_{\a=1}^{m} \mu_\a |v_\a|^2,
\end{equation}
where
$$\mu_1\geq \mu_2\geq \cdots\geq \mu_N$$
are eigenvalues of $\Phi$ and $\mu_\a:=0$ whenever $\a\geq N+1$.
The equality holds if and only if there exists an orthonormal basis $\ep_1,\cdots,\ep_N$ of $V$,
such that
\begin{itemize}
\item $\Phi(\ep_i)=\mu_i \ep_i$ for each $1\leq i\leq N$;
\item For any $s>0$, denote by $\be$ the maximal index satisfying $|v_\be|\geq s$,
then $\text{span}\{v_1,\cdots,v_\be\}=\text{span}\{\ep_1,\cdots,\ep_\be\}$.

\end{itemize}
\end{lem}

\begin{proof}
Let
$$t_1\geq t_2\geq \cdots\geq t_m$$
be any fixed numbers, satisfying $t_k>0$ and $t_{k+1}=0$, due to the compactness,
there exists an orthonormal set $e_1,\cdots,e_k$ ($|e_1|=\cdots=|e_k|=1$
and $\lan e_i,e_j\ran= 0$ whenever $i\neq j$), such that
\begin{equation}\label{max}
\sum_{\a=1}^m \lan \Phi(v_\a),v_\a\ran=\sum_{\a=1}^k \lan \Phi(v_\a),v_\a\ran\leq \sum_{\a=1}^k \lan \Phi(t_\a e_\a),t_\a e_\a\ran
\end{equation}
for each pairwise orthogonal vectors $v_1,\cdots, v_m$ satisfying
\begin{equation}
|v_1|=t_1,\cdots, |v_m|=t_m.
\end{equation}
We shall study the properties of $e_1,\cdots,e_k$. 

Let $v$ be an arbitrary unit vector which are orthogonal to each $e_\a$ ($1\leq \a\leq k$),
then (\ref{max}) implies
$$\aligned
0=&\f{d}{d\th}\Big|_{\th=0}\Bigg(\sum_{\a\neq \be}\lan \Phi(t_\a e_\a),t_\a e_\a\ran+\lan \Phi(t_\be(\cos\th e_\be+\sin\th v)),t_\be(\cos\th e_\be+\sin\th v)\ran\Bigg)\\
=&\lan \Phi(t_\be v),t_\be e_\be\ran+\lan \Phi(t_\be e_\be),t_\be v\ran=2 t_\be^2\lan \Phi(e_\be),v\ran.
\endaligned$$
This means
\begin{equation}
\Phi(e_\a)\in \text{span}\{e_1,\cdots,e_k\}\qquad \forall \a=1,\cdots, k.
\end{equation}

For any $\be,\g$ satisfying $t_\be\neq t_\g$, again using (\ref{max}), we can derive
$$\aligned
0=&\f{d}{d\th}\Big|_{\th=0}\Bigg(\sum_{\a\neq \be,\g}\lan \Phi(t_\a e_\a),t_\a e_\a\ran+\lan \Phi(t_\be(\cos\th e_\be+\sin\th e_\g)),t_\be(\cos\th e_\be+\sin\th e_\g)\ran\\
&+\lan \Phi(t_\g(-\sin\th e_\be+\cos\th e_\g)),t_\g(-\sin\th e_\be+\cos\th e_\g)\ran\Bigg)\\
=&2(t_\be^2-t_\g^2)\lan \Phi(e_\be),e_\g\ran.
\endaligned$$
Hence for each $t\in \{t_1,\cdots,t_k\}$,
\begin{equation}
V_{t}:=\text{span}\{e_\a:t_\a=t\}
\end{equation}
is invariant under $\Phi$. Noting that $\Phi|_{V_t}$ is still self-adjoint, there exists an orthonormal basis $\{\ep_\a:t_\a=t\}$
of $V_t$, such that
\begin{equation}
\Phi(\ep_\a)=\hat{\mu}_\a \ep_\a,
\end{equation}
where $\{\hat{\mu}_\a:t_\a=t\}$ are eigenvalues of $\Phi|_{V_t}$ arranged in decreasing order,
and then
\begin{equation}\label{phi1}\aligned
\sum_{t_\a=t}\lan \Phi(t_\a e_\a),t_\a e_\a\ran=& t^2\sum_{t_\a=t}\lan \Phi(e_\a),e_\a\ran=t^2 (\text{tr }\Phi|_{V_t})\\
=& t^2\sum_{t_\a=t}\lan \Phi(\ep_\a),\ep_\a\ran=\sum_{t_\a=t}t_\a^2\hat{\mu}_\a
\endaligned
\end{equation}
Therefore
\begin{equation}
\sum_{\a=1}^k \lan \Phi(t_\a e_\a),t_\a e_\a\ran=\sum_{\a=1}^k t_\a^2 \hat{\mu}_\a.
\end{equation}
Then the definition of $e_1,\cdots,e_k$ forces $\mu_\a=\hat{\mu}_\a$ for each $1\leq \a\leq k$.
Extend $\ep_1,\cdots,\ep_k$ to an orthonormal basis $\ep_1,\cdots,\ep_N$
satisfying $\Phi(\ep_i)=\mu_i \ep_i$,
then
\begin{equation}
\text{span}\{v_1,\cdots,v_\be\}=\text{span}\{e_1,\cdots,e_\be\}=\text{span}\{\ep_1,\cdots,\ep_\be\}
\end{equation}
with $v_\a:=t_\a e_\a$
for each $\be$ which is the maximal index satisfying $|v_\be|=t_\be\geq s$ for any given $s>0$.
On the other hand, whenever $v_1,\cdots,v_m$ satisfy the above condition,
a similar calculation as in (\ref{phi1}) shows $\sum\limits_{\a=1}^m \lan \Phi(v_\a),v_\a\ran=\sum\limits_{\a=1}^{m} \mu_\a |v_\a|^2$.
\end{proof}


Assume $W$ is an $m$-dimensional inner product space and $F:W\ra V$ is a linear mapping,
then it is easy to check that, for any two orthonormal basis $u_1,\cdots,u_m$ and
$u'_1,\cdots,u'_m$ of $W$, we have
\begin{equation}
\sum_{\a=1}^m \lan \Phi(F(u_\a)),F(u_\a)\ran=\sum_{\a=1}^m \lan \Phi(F(u'_\a)),F(u'_\a)\ran.
\end{equation}
Denote by
\begin{equation}
S:=\Big(\lan F(u_\a),F(u_\be)\ran\Big)
\end{equation}
the \textit{fundamental matrix} of $F$ w.r.t. $u_1,\cdots,u_m$, then
any two fundamental matrices of $F$ are orthogonal similar to each other,
which have the same eigenvalues. Thus, without loss of generality we can assume
\begin{equation}
S=\text{diag}(\la_1,\cdots,\la_m)
\end{equation}
with
\begin{equation}
\la_1\geq \cdots\geq \la_m\geq 0.
\end{equation}
This implies $v_1:=F(u_1),\cdots,v_m:=F(u_m)$ are vectors in $V$ satisfying the conditions of Lemma \ref{eig}.
Applying this lemma, we can derive a corollary as follows.

\begin{cor}\label{eig1}
Let $F$ be a linear map from an $m$-dimensional inner product space $W$ into an $N$-dimensional inner product space $V$
and $\Phi:V\ra V$ be a non-negative definite self-adjoint transformation,
then
\begin{equation}
\sum_{\a=1}^m\lan \Phi(F(u_\a)),F(u_\a)\ran\leq \sum_{\a=1}^m \mu_\a \la_\a,
\end{equation}
where $\mu_1\geq \cdots\geq \mu_N$ are eigenvalues of $\Phi$ $(\mu_\a:=0$ for each $\a\geq N+1)$
and $\la_1\geq \cdots \geq \la_m$ are eigenvalues of the fundamental matrix of $F$.
The equality holds if and only if there exists an orthonormal basis $u_1,\cdots,u_m$ of $W$
and an orthonormal basis $\ep_1,\cdots,\ep_N$ of $V$, such that
\begin{itemize}
\item $\lan F(u_\a),F(u_\be)\ran=\la_\a \de_{\a\be}$ for each $1\leq \a,\be\leq m$;
\item $\Phi(\ep_i)=\mu_i\ep_i$ for each $1\leq i\leq N$;
\item $F(u_\a)=\sqrt{\la_\a}\ep_\a$ whenever $\la_\a>0$.
\end{itemize}
\end{cor}

\begin{proof}
It remains to consider the cases when the equality holds.
Denote by $\lan\cdot,\cdot\ran_W$ be the inner product on $W$,
then there exists a linear transformation $\Psi$ on $W$, such that
\begin{equation}
\lan F(u_0),F(u)\ran=\lan \Psi(u_0),u\ran_W
\end{equation}
for any $u_0,u\in W$. It is easy to check that $\Psi$ is self-adjoint,
and for each $\la\in \R$, $\la$ is an eigenvalue of $\Psi$ if and only if
$\la$ is an eigenvalue of the fundamental matrix of $F$. By Lemma \ref{eig},
there exists an orthonormal basis $\ep_1,\cdots,\ep_N$ of $V$, such that
$\Phi(\ep_i)=\mu_i \ep_i$, and for each positive eigenvalue $\la$ of $\Psi$,
\begin{equation}
F(W_\la)=\text{span}\{\ep_{l+1},\cdots,\ep_{l+r}\},
\end{equation}
where $W_\la$ is the eigenspace associated to $\la$, $l$ is the number of eigenvalues of
$\Psi$ which are strictly greater than $\la$ (counting the multiplicities) and $r:=\dim W_\la$.
This enable us to find an orthonormal basis $u_{l+1},\cdots,u_{l+r}$ of $W_\la$,
such that $F(u_\a)=\sqrt{\la}\ep_\a$ for each $l+1\leq \a\leq l+r$.
Collecting all such vectors along with an arbitrary orthonormal basis of $W_0$,
we get an orthnormal basis $u_1,\cdots,u_m$ of $W$ that is the required one.

\end{proof}

\subsection{A lemma on squared differences}

The second lemma gives a lower bound for the partial sums of squared differences.

\begin{lem}\label{sq-diff}
Let $\eta_1,\cdots,\eta_n$ be real numbers satisfying
\begin{equation}
\sum_{i=1}^n \eta_i=0,\quad\sum_{i=1}^n \eta_i^2=1.
\end{equation}
Denote
\begin{equation}
\{(\eta_i-\eta_j)^2:1\leq i\leq j\leq n\}=\{s_1,\cdots,s_N\},
\end{equation}
where $N:=\f{n(n+1)}{2}$ and
\begin{equation}
s_1\geq s_2\geq \cdots \geq s_N,
\end{equation}
and let
\begin{equation}
S_k:=s_1+s_2+\cdots+s_k,
\end{equation}
then
\begin{equation}
S_k\leq \min\{k+1,n\}.
\end{equation}
Moreover, for each $1\leq k\leq n-1$, $S_k=k+1$ if and only if $\{\eta_1,\cdots,\eta_n\}$ contains
exactly $(k+1)$ nonzero numbers, of which $k$ ones have the same value.

\end{lem}

\begin{proof}
Firstly,
\begin{equation}
S_N=\sum_{1\leq i\leq j\leq n}(\eta_i-\eta_j)^2=n\sum_{i=1}^n \eta_i^2-\left(\sum_{i=1}^n \eta_i\right)^2=n
\end{equation}
implies $S_k\leq n$ for each $1\leq k\leq N$. So it is sufficient for us to prove
$S_k\leq k+1$ for each $1\leq k\leq n-1$. We shall do it by induction on $k$.

Without loss of generality we can assume
\begin{equation}
\eta_1\geq \eta_2\geq \cdots \geq\eta_n,
\end{equation}
so
\begin{equation}\aligned
S_1=&s_1=(\eta_1-\eta_n)^2=\eta_1^2+\eta_n^2-2\eta_1\eta_n\\
\leq &2(\eta_1^2+\eta_n^2)\leq 2,
\endaligned
\end{equation}
where the equality holds if and only if $\eta_1=\f{1}{\sqrt{2}}$, $\eta_2=\cdots=\eta_{n-1}=0$, and $\eta_n=-\f{1}{\sqrt{2}}$.

If $s_2\leq 1$, then by the inductive assumption, we have
\begin{equation}
S_2=S_1+s_2\leq 2+1=3.
\end{equation}
Note that here the equality cannot hold, since $S_1=2$ and $s_2=1$ could not happen simultaneously.
Otherwise, $s_2>1$ means either $\eta_1-\eta_{n-1}>1$ or $\eta_2-\eta_n>1$.
Without loss of generality we can assume
\begin{equation}
\eta_1-\eta_{n-1}>1.
\end{equation}
(If not, it suffices to consider $-\eta_n,\cdots,-\eta_1$ in stead of $\eta_1,\cdots,\eta_n$.) Thus
\begin{equation}
\aligned
S_2=&(\eta_1-\eta_n)^2+(\eta_1-\eta_{n-1})^2=2\eta_1^2+\eta_n^2+\eta_{n-1}^2-2\eta_1(\eta_n+\eta_{n-1})\\
\leq &2\eta_1^2+\eta_n^2+\eta_{n-1}^2+\eta_1^2+(\eta_n+\eta_{n-1})^2\leq 3\eta_1^2+\eta_n^2+\eta_{n-1}^2+2(\eta_n^2+\eta_{n-1}^2)\\
=&3(\eta_1^2+\eta_n^2+\eta_{n-1}^2)\leq 3,
\endaligned
\end{equation}
where the equality holds if and only if $\eta_1=\f{2}{\sqrt{6}}$, $\eta_2=\cdots=\eta_{n-2}=0$, and $\eta_{n-1}=\eta_n=-\f{1}{\sqrt{6}}$.

Now we assume $S_{k-1}\leq k$ with $k\geq 3$. If $s_k\leq 1$, then
\begin{equation}
S_k=S_{k-1}+s_k\leq k+1
\end{equation}
and the equality cannot hold for the same reason as above. Otherwise, we claim
\begin{equation}
s_2=(\eta_1-\eta_{n-1})^2, s_3=(\eta_1-\eta_{n-2})^2, \cdots, s_k=(\eta_1-\eta_{n-k+1})^2.
\end{equation}
This is a direct corollary of
$$(\eta_1-\eta_{n-1})^2+(\eta_2-\eta_n)^2\leq 2(\eta_1^2+\eta_{n-1}^2+\eta_2^2+\eta_n^2)\leq 2$$
(note that here $n\geq 4$). Hence
\begin{equation}
\aligned
S_k=&\sum_{\a=0}^{k-1}(\eta_1-\eta_{n-\a})^2=k\eta_1^2+\sum_{\a=0}^{k-1}\eta_{n-\a}^2-2\eta_1\left(\sum_{\a=0}^{k-1}\eta_{n-\a}\right)\\
=&k\eta_1^2+\sum_{\a=0}^{k-1}\eta_{n-\a}^2+\eta_1^2+\left(\sum_{\a=0}^{k-1}\eta_{n-\a}\right)^2\leq (k+1)\eta_1^2+\sum_{\a=0}^{k-1}\eta_{n-\a}^2+k\sum_{\a=0}^{k-1}\eta_{n-\a}^2\\
=&(k+1)\left(\eta_1^2+\sum_{\a=0}^{k-1}\eta_{n-\a}^2\right)\leq k+1,
\endaligned
\end{equation}
where the equality holds if and only if $\eta_1=\f{k}{\sqrt{k(k+1)}}$, $\eta_2=\cdots=\eta_{n-k}=0$
and $\eta_{n-k+1}=\cdots=\eta_{n-1}=\eta_n=-\f{1}{\sqrt{k(k+1)}}$.

\end{proof}

\subsection{A more refined inequality}

\begin{thm}\label{ineq1}
Let $A,A_2,\cdots,A_m$ be $(n\times n)$-symmetric matrices, such that
\begin{itemize}
\item $\|A\|=1$, $\text{tr}(A)=0$;
\item $\lan A_\a,A_\be\ran=0$ for each $2\leq \a< \be\leq m$;
\item $\|A_2\|\geq \cdots\geq \|A_m\|$.
\end{itemize}
Then
\begin{equation*}
\sum_{\a=2}^m\left\|[A,A_\a]\right\|^2\leq \|A_2\|^2+\sum_{\a=2}^{\min\{m,n\}}\|A_\a\|^2.
\end{equation*}
\end{thm}

\begin{proof}
It is easy to check that
$$\lan U^T BU,U^TCU\ran=\lan B,C\ran$$
and
$$[U^TBU,U^TCU]=U^T[B,C]U$$
for any 2 symmetric matrices $B,C$ and an arbitrary orthogonal matrix $U$. Hence we can assume $A$
is a diagonal matrix without loss of generality, i.e.
\begin{equation}
A=\text{diag}(\eta_1,\cdots,\eta_n)
\end{equation}
with
\begin{equation}
\sum_{i=1}^n \eta_i=0,\quad\sum_{i=1}^n \eta_i^2=1.
\end{equation}

For any symmetric $(n\times n)$-matrix $B$,
\begin{equation}\
\|[A,B]\|^2=\lan \operatorname{ad} A(B),\operatorname{ad} A(B)\ran=\lan (\operatorname{ad} A)^2(B),B\ran.
\end{equation}
A straightforward calculation shows
$$(\operatorname{ad} A)^2 (E_{ij}+E_{ji})=(\eta_i-\eta_j)^2 (E_{ij}+E_{ji})$$
for each $1\leq i,j\leq n$. This means $(\operatorname{ad} A)^2$ is a self-adjoint nonnegative definite transformation on the vector space $\text{Sym}_n$
of all symmetric $(n\times n)$-matrices, whose eigenvalues are
\begin{equation}
(\eta_i-\eta_j)^2 \qquad \text{with }1\leq i\leq j\leq n.
\end{equation}
We rearrange these numbers in decreasing order, denoted by $s_1,\cdots, s_N$,
as in Lemma \ref{sq-diff}. If $m-1>N$, additionally define $s_{N+1}=\cdots=s_{m-1}=0$.
Then applying Lemma \ref{eig} and Lemma \ref{sq-diff} gives
\begin{equation}\label{Lie-ineq}\aligned
&\sum_{\a=2}^m \|[A,A_\a]\|^2=\sum_{\a=2}^m \lan (\operatorname{ad} A)^2(A_\a),A_\a\ran\leq \sum_{\a=2}^{\min\{m,N+1\}}s_{\a-1}\|A_\a\|^2\\
=&\sum_{\a=2}^m s_{\a-1}\|A_\a\|^2=\sum_{\a=2}^{m-1}S_{\a-1}(\|A_\a\|^2-\|A_{\a+1}\|^2)+S_{m-1}\|A_m\|^2\\
\leq & \sum_{\a=2}^{m-1}\min\{\a,n\}(\|A_\a\|^2-\|A_{\a+1}\|^2)+\min\{m,n\}\|A_m\|^2\\
=&\sum_{\a=2}^{\min\{m,n\}}\|A_\a\|^2+\|A_2\|^2.
\endaligned
\end{equation}

\end{proof}

Let
\begin{equation}
\text{Sym}_n^0:=\{B\in \text{Sym}_n:\text{tr }B=0\}
\end{equation}
 be the vector space of all traceless symmetric $(n\times n)$-matrices.
Assume $W$ is an $m$-dimensional vector space and $F:W\ra \text{Sym}_n^0$ is a linear mapping.
Denote by
\begin{equation}
\la_1\geq \la_2\geq \cdots\geq \la_m\geq 0
\end{equation}
the eigenvalues of the fundamental matrix of $F$ w.r.t. an arbitrary orthonormal basis of $W$.
Let $u_1$ be a unit vector in $W$ satisfying $\lan A_1,A_1\ran=\la_1$ with $A_1:=F(u_1)$,
$u_2,\cdots,u_m$ be an orthonormal basis of $u_1^\bot$ (the orthogonal complement space of $\R u_1$)
and $A_\a:=F(u_\a)$ for each $2\leq \a\leq m$, then
\begin{equation}
\sum_{\a=2}^m\left\|[A_1,A_\a]\right\|^2=\sum_{\a=2}^m \lan (\operatorname{ad}A_1)^2 F(u_\a),F(u_\a)\ran
\end{equation}
is independent of the choices of $u_2,\cdots,u_m$. Thereby, without loss of generality
we can assume $\lan A_\a,A_\be\ran=\la_\a\de_{\a\be}$ for each $2\leq \a,\be\leq m$,
and then Theorem \ref{ineq1} gives rise to the following conclusion:
\begin{thm}
Let $F$ be a linear map from an $m$-dimensional inner product space $W$ into $\text{Sym}_n^0$,
$\la_1\geq \la_2\geq \cdots \geq \la_m$ be the eigenvalues of the fundamental matrix of $F$.
Assume $u_1,\cdots,u_m$ is an orthonormal basis of $W$, such that $\lan A_1,A_1\ran=\la_1$
with $A_1:=F(u_1)$, then
\begin{equation}\label{ineq2}
\sum_{\a=2}^m\left\|[A_1,A_\a]\right\|^2\leq \la_1\left(\la_2+\sum_{\a=2}^{\min\{m,n\}}\la_\a\right)
\end{equation}
with $A_\a:=F(u_\a)$.
\end{thm}

It is subtle to clarify the cases when (\ref{ineq2}) holds. Now we only consider the situation of $n=3$.
Note that in the sequel $\la_\a:=0$ for each $\a>m$ and we assume $\la_1>0$ (since $\la_1=0$ is a trivial case).

As shown above, we can assume $A_1=\mu_1 A$, where $\mu_1:=\sqrt{\la_1}$ and
\begin{equation}
A=\text{diag}(\eta_1,\eta_2,\eta_3).
\end{equation}
A direct calculation shows $S_3=3$. Thereby, the equality of the second inequality
in (\ref{Lie-ineq}) holds if and only if
\begin{equation}\label{con1}\aligned
&S_1(\|A_2\|^2-\|A_3\|^2)=2(\|A_2\|^2-\|A_3\|^2)\text{ and }\\
&S_2(\|A_3\|^2-\|A_4\|^2)=3(\|A_3\|^2-\|A_4\|^2).
\endaligned
\end{equation}

In the following, we shall discuss case by case according to the eigenvalues of the fundamental matrix of $F$.

{\bf Case I:  $\la_3>\la_4$.} Thereby (\ref{con1}) shows $S_2=3$. By Lemma \ref{sq-diff},
we can assume
\begin{equation}
\eta_1=\f{2}{\sqrt{6}},\quad \eta_2=\eta_3=-\f{1}{\sqrt{6}}
\end{equation}
without loss of generality. This means
\begin{equation}
s_1=s_2=\f{3}{2},\quad s_3=0.
\end{equation}
Again using (\ref{con1}), we obtain $\la_2=\la_3$.
Denote by $V_s\subset \text{Sym}_n^0$ the eigenspace of $(\operatorname{ad} A)^2$ associated to the eigenvalue $s$,
then Corollary \ref{eig1} shows the equality of (\ref{ineq2}) holds if and only if
$F$ maps $W_{\la_2}$ onto $V_{\f{3}{2}}$ and $F(W_{\la_\a})\subset V_0$ for
each $\a\geq 4$, where
\begin{equation}
V_{\f{3}{2}}=\text{span}\{E_{12}+E_{21},E_{13}+E_{31}\}
\end{equation}
and
\begin{equation}
V_0=\text{span}\{E_{23}+E_{32},\text{diag}(0,1,-1)\}.
\end{equation}
Equivalently, there exists an orthonormal basis $u_1,\cdots,u_m$ of $W$ and an orthogonal matrix
 $$T=\begin{pmatrix}1 & & \\
 & \cos\th & -\sin\th\\
 & \sin\th & \cos\th
 \end{pmatrix}$$
 with suitable $\th$, such that
\begin{equation}\aligned
T^tA_1T=\f{\mu_1}{\sqrt{6}}\text{diag}(2,-1,-1),
\quad&
T^tA_2T=\f{\mu_2}{\sqrt{2}}(E_{12}+E_{21})\\
T^tA_3T=\f{\mu_2}{\sqrt{2}}(E_{13}+E_{31}),\quad&
T^tA_4T=\f{\mu_4}{\sqrt{2}}(E_{23}+E_{32}),\\
T^tA_5T=\f{\mu_5}{\sqrt{2}}\text{diag}(0,-1,1),\quad&
T^tA_\a T=0\ (\forall \a\geq 6)
\endaligned
\end{equation}
where $\mu_\a:=\sqrt{\la_\a}$, $T^t$ denotes the transpose of $T$,
and $A_\a:=F(u_\a)$.

{\bf Case II:  $\la_2>\la_3$.} Thereby (\ref{con1}) shows $S_1=2$. By Lemma \ref{sq-diff},
without loss of generality we can assume
\begin{equation}
\eta_1=\f{1}{\sqrt{2}},\ \eta_2=-\f{1}{\sqrt{2}},\ \eta_3=0.
\end{equation}
This implies
\begin{equation}
s_1=2,\quad s_2=s_3=\f{1}{2}.
\end{equation}
Again applying (\ref{con1}), we have $\la_3=\la_4$. In conjunction of Corollary \ref{eig1},
the equality of (\ref{ineq2}) holds if and only if
$F$ maps $W_{\la_2}$ onto $V_{2}$, $F(W_{\la_3})\subset V_{\f{1}{2}}$ and $F(W_{\la_\a})\subset V_0$
for each $\a\geq 5$, where
\begin{equation}\aligned
V_2&=\text{span}\{E_{12}+E_{21}\},\\
V_{\f{1}{2}}&=\text{span}\{E_{13}+E_{31},E_{23}+E_{32}\},\\
V_0&=\text{span}\{\text{diag}(-1,-1,2)\}.
\endaligned
\end{equation}
 Equivalently, there exists an orthonormal basis $u_1,\cdots,u_m$ of $W$, such that
\begin{equation}\aligned
A_1=\f{\mu_1}{\sqrt{2}}\text{diag}(1,-1,0),
\quad&
A_2=\f{\mu_2}{\sqrt{2}}(E_{12}+E_{21})\\
A_3=\f{\mu_3}{\sqrt{2}}(E_{13}+E_{31}),\quad&
A_4=\f{\mu_3}{\sqrt{2}}(E_{23}+E_{32}),\\
A_5=\f{\mu_5}{\sqrt{6}}\text{diag}(-1,-1,2),\quad&
A_\a=0\ (\forall \a\geq 6)
\endaligned
\end{equation}

{\bf Case III:  $\la_2=\la_3=\la_4$.} In this case, (\ref{con1}) automatically holds.
Without loss of generality we can assume
\begin{equation}
\eta_1\geq 0\geq \eta_3\geq \eta_2.
\end{equation}
If $0>\eta_3>\eta_2$,
the equality of (\ref{ineq2}) holds if and only if
$F(W_{\la_2})\subset V_{(\eta_1-\eta_2)^2}\oplus V_{(\eta_1-\eta_3)^2}\oplus V_{(\eta_2-\eta_3)^2}$
and $F(W_{\la_\a})\subset V_0$ for each $\a\geq 5$, where
\begin{equation}\aligned
&V_{(\eta_1-\eta_2)^2}=\text{span}\{E_{12}+E_{21}\},\\
&V_{(\eta_1-\eta_3)^2}=\text{span}\{E_{13}+E_{31}\},\\
&V_{(\eta_2-\eta_3)^2}=\text{span}\{E_{23}+E_{32}\},\\
&V_0=\text{span}\{\text{diag}(\eta_2-\eta_3,\eta_3-\eta_1,\eta_1-\eta_2)\}.
\endaligned
\end{equation}
Equivalently, there exists an orthonormal basis $u_1,\cdots,u_m$ of $W$, such that
\begin{equation}\aligned
A_1=&\mu_1\text{diag}(\eta_1,\eta_2,\eta_3),
\quad
A_2=\f{\mu_2}{\sqrt{2}}(E_{12}+E_{21})\\
A_3=&\f{\mu_2}{\sqrt{2}}(E_{13}+E_{31}),\quad
A_4=\f{\mu_2}{\sqrt{2}}(E_{23}+E_{32}),\\
A_5=&\f{\mu_5}{\sqrt{3}}\text{diag}(\eta_2-\eta_3,\eta_3-\eta_1,\eta_1-\eta_2),\\
A_\a=&0\quad (\forall \a\geq 6).
\endaligned
\end{equation}
When $\eta_2=\eta_3$ (i.e. $(\eta_1,\eta_2,\eta_3)=\f{1}{\sqrt{6}}(2,-1,-1)$) or
$\eta_3=0$ (i.e. $(\eta_1,\eta_2,\eta_3)=\f{1}{\sqrt{2}}(1,-1,0)$),
the discussion is quite similar as in Case I or II, respectively, yielding the
analogous results. So we omit the details here.

As a summary, we rewrite the conclusions in the following proposition.

\begin{pro}\label{dim3}
When $n=3$, the equality of (\ref{ineq2}) holds if and only if there exists
an orthonormal basis $u_1,\cdots,u_m$, such that
after an orthonormal base change and up to a sign, one and only one of following 3 cases holds:
\begin{itemize}
\item {(I)} $A_1=\f{\mu_1}{\sqrt{6}}\mathrm{diag}(2,-1,-1)$, $A_2=\f{\mu_2}{\sqrt{2}}(E_{12}+E_{21})$,
$A_3=\f{\mu_2}{\sqrt{2}}(E_{13}+E_{31})$, $A_4=\f{\mu_4}{\sqrt{2}}(E_{23}+E_{32})$,
$A_5=\f{\mu_5}{\sqrt{2}}\mathrm{diag}(0,-1,1)$ and $A_\a=0$ for each $\a\geq 6$,
with $\mu_1\geq \mu_2>\mu_4\geq \mu_5$.
\item {(II)} $A_1=\f{\mu_1}{\sqrt{2}}\mathrm{diag}(1,-1,0)$, $A_2=\f{\mu_2}{\sqrt{2}}(E_{12}+E_{21})$,
$A_3=\f{\mu_3}{\sqrt{2}}(E_{13}+E_{31})$, $A_4=\f{\mu_3}{\sqrt{2}}(E_{23}+E_{32})$,
$A_5=\f{\mu_5}{\sqrt{6}}\mathrm{diag}(-1,-1,2)$ and $A_\a=0$ for each $\a\geq 6$,
with $\mu_1\geq \mu_2>\mu_3\geq \mu_5$.
\item {(III)} $A_1=\mu_1\mathrm{diag}(\eta_1,\eta_2,\eta_3)$, $A_2=\f{\mu_2}{\sqrt{2}}(E_{12}+E_{21})$,
$A_3=\f{\mu_2}{\sqrt{2}}(E_{13}+E_{31})$, $A_4=\f{\mu_2}{\sqrt{2}}(E_{23}+E_{32})$,
$A_5=\f{\mu_5}{\sqrt{3}}\mathrm{diag}(\eta_2-\eta_3,\eta_3-\eta_1,\eta_1-\eta_2)$ and $A_\a=0$ for each $\a\geq 6$,
with $\mu_1\geq \mu_2\geq \mu_5$ and $\eta_1\geq 0\geq \eta_3\geq \eta_2$.
\end{itemize}
Here $A_\a:=F(u_\a)$ and $\mu_\a:=\sqrt{\la_\a}$.
\end{pro}

\bigskip\bigskip

    \Section{Simons-type inequality for minimal submanifolds in $S^{n+m}$}{Simons-type inequality for minimal submanifolds in $S^{n+m}$}\label{preliminaries}

    \subsection{Minimal submanifolds in the unit spheres}

Let $M$ be an $n$-dimensional submanifold in an $(n+m)$-dimensional Riemannian manifold $\ol{M}$, equipped with the induced Riemannian metric. The {\it second fundamental form} $B$ is
a pointwise symmetric bilinear form on $T_p M$ (the tangent space at $p\in M$) with values in $N_p M$ (the normal space at $p$), defined by
\begin{equation}
B_{XY}=(\ol{\n}_X Y)^N
\end{equation}
with $\ol{\n}$ the induced Levi-Civita connection on $\ol{M}$.  The induced connection on the tangent bundle $TM$ and the normal bundle $NM$ are defined by
\begin{equation}
\n_X Y=(\ol{\n}_X Y)^T,\quad \n_X \nu=(\ol{\n}_X \nu)^N.
\end{equation}
Here $X,Y$ are local smooth sections of $TM$, $\nu$ is a local smooth section of $NM$,
$(\cdots)^T$ and $(\cdots)^N$ denote the orthogonal projection onto the tangent space and the normal space, respectively.
The {\it shape operator}
\begin{equation}
A^\nu(X):=-(\ol{\n}_X \nu)^T,
\end{equation}
satisfying the Weingarten equations
\begin{equation}
\lan A^\nu(X),Y\ran=\lan B_{XY},\nu\ran,
\end{equation}
is a self-adjoint transformation on the tangent space at each point. The second fundamental form,
the curvature tensor of the submanifold, the curvature tensor of the normal bundle
and the curvature tensor of the ambient manifold satisfy the {\it Gauss equations},
the {\it Codazzi equations} and the {\it Ricci equations} (see e.g. \cite{X}):
\begin{equation}\label{GCR}
\aligned
&\lan R_{XY}Z,W\ran=\lan \ol{R}_{XY}Z,W\ran+\lan B_{XZ},B_{YW}\ran-\lan B_{XW},B_{YZ}\ran,\\
 &(\n_X B)_{YZ}-(\n_Y B)_{XZ}=-(\ol{R}_{XY}Z)^N,\\
&\lan R_{XY}\mu,\nu\ran=\lan \ol{R}_{XY}\mu,\nu\ran+\lan A^\mu(X),A^\nu(Y)\ran-\lan A^\mu(Y),A^\nu(X)\ran.
\endaligned
\end{equation}
Here
\begin{equation}\aligned
&R_{XY}Z:=-\n_X\n_Y Z+\n_Y\n_X Z+\n_{[X,Y]}Z,\\
&\ol{R}_{XY}Z:=-\ol{\n}_X\ol{\n}_Y Z+\ol{\n}_Y\ol{\n}_X Z+\ol{\n}_{[X,Y]}Z,\\
&(\n_X B)_{YZ}:=\n_X B_{YZ}-B_{\n_X Y,Z}-B_{Y,\n_X Z},\\
&R_{XY}\mu:=-\n_X\n_Y \mu+\n_Y\n_X \mu+\n_{[X,Y]}\mu.
\endaligned
\end{equation}
Denote
\begin{equation}
(\n^2_{XY}B):=\n_X\n_Y B-\n_{\n_X Y}B,
\end{equation}
then
\begin{equation}\label{dXY}
\n_{XY}^2 B-\n_{YX}^2 B=\n_X\n_Y B-\n_Y\n_X B-\n_{[X,Y]}B=-R_{XY}B.
\end{equation}

Taking the trace of $B$ gives the {\it mean curvature vector} $H$ of $M$.
$M$ is called a {\it minimal submanifold} whenever $H$ vanishes everywhere on $M$.
Particularly if $B\equiv 0$, $M$ is called a {\it totally geodesic submanifold}.

Now we assume $M$ is a minimal submanifold in $S^{n+m}$, the unit Euclidean sphere of dimension $(n+m)$.
Let $\{E_1,\cdots,E_n\}$ be a local orthonormal tangent frame field, $\{\nu_1,\cdots,\nu_m\}$
be a local orthonormal normal field on $M$, and
\begin{equation}
h_{ij}^\a:=\lan B_{E_iE_j},\nu_\a\ran,\quad h_{ijk}^\a:=\lan (\n_{E_k}B)_{E_iE_j},\nu_\a\ran,\quad h_{ijkl}^\a=\lan (\n^2_{E_lE_k}B)_{E_iE_j},\nu_\a\ran
\end{equation}
be the coefficients of $B$, $\n B$ and $\n^2 B$, respectively. We shall use the summation convention and agree on the ranges of indices
\begin{equation}
1\leq i,j,k,l\leq n,\quad 1\leq \a,\be,\g,\de\leq m.
\end{equation}
Since $S^{n+m}$ has constant sectional curvature $1$,
\begin{equation}\label{s1}
\lan \ol{R}_{XY}Z,W\ran=\lan X,Z\ran\lan Y,W\ran-\lan X,W\ran \lan Y,Z\ran
\end{equation}
and
\begin{equation}\label{s2}
(\ol{R}_{XY}Z)^N=\lan \ol{R}_{XY}\mu,\nu\ran=0.
\end{equation}
As in \cite{Lu}, define
\begin{equation}
\De h_{ij}^\a:=\sum_{k=1}^n h_{ijkk}^\a,
\end{equation}
then a direct calculation based on (\ref{GCR}), (\ref{dXY}), (\ref{s1}), (\ref{s2}) and the minimality of $M$ shows
\begin{equation}\label{LaB}\aligned
\De h_{ij}^\a=&h_{ijkk}^\a=h_{kijk}^\a=(h_{kijk}^\a-h_{kikj}^\a)+h_{kkij}^\a=-\lan (R_{E_k E_j}B)_{E_kE_i},\nu_\a\ran\\
=&-\lan R_{E_k E_j}B_{E_kE_i},\nu_\a\ran+\lan B_{R_{E_k E_j}E_k,E_i},\nu_\a\ran+\lan B_{E_k,R_{E_kE_j}E_i},\nu_\a\ran\\
=&-h_{ki}^\be\big(\lan A^\be(E_k),A^\a(E_j)\ran-\lan A^\be(E_j),A^\a(E_k)\ran\big)\\
&+h_{li}^\a(\de_{kk}\de_{jl}-\de_{kl}\de_{jk}+h_{kk}^\be h_{jl}^\be-h_{kl}^\be h_{jk}^\be)\\
&+h_{kl}^\a(\de_{ki}\de_{jl}-\de_{kl}\de_{ji}+h_{ki}^\be h_{jl}^\be-h_{kl}^\be h_{ji}^\be)\\
=&nh_{ij}^\a-h_{ik}^\be h_{kl}^\be h_{lj}^\a+2h_{ik}^\be h_{kl}^\a h_{lj}^\be-h_{il}^\a h_{lk}^\be h_{kj}^\be-h_{kl}^\be h_{kl}^\a h_{ij}^\be.
\endaligned
\end{equation}
Here $A^\a:=A^{\nu_\a}$ is the shape operator on the tangent space at each considered point, which can also be seen as
an $(n\times n)$-symmetric matrix, i.e. $A^\a=(h_{ij}^\a)$. Thus, in terms of matrix notations, (\ref{LaB}) can be rewritten as (see \cite{Lu})
\begin{equation}\label{LaB1}
\De A^\a=nA^\a-\lan A^\a,A^\be\ran A^\be-[A^\be,[A^\be,A^\a]].
\end{equation}

Let
\begin{equation}
|B|^2:=\lan B_{E_iE_j},B_{E_iE_j}\ran=\sum_{i,j,\a}(h_{ij}^\a)^2
\end{equation}
be the squared norm of the second fundamental form, then
\begin{equation}\label{LaB2}\aligned
&\f{1}{2}\De |B|^2=h_{ijk}^\a h_{ijk}^\a+h_{ij}^\a h_{ijkk}^\a\\
=&|\n B|^2+n|B|^2-\sum_{\a,\be} \lan A^\a,A^\be\ran^2-\sum_{\a,\be}\|[A^\a,A^\be]\|^2
\endaligned
\end{equation}
with $\De$ the Laplace-Beltrami operator on $M$.

\subsection{A pinching theorem on eigenvalues of fundamental matrices}

Let
\begin{equation}
S(\mu,\nu):=\lan A^\mu,A^\nu\ran=\sum_{i,j}\lan A^\mu(E_i),E_j\ran \lan A^\nu(E_i),E_j\ran
\end{equation}
be a bilinear nonnegative definite symmetric tensor on $M$,
\begin{equation}
(S_{\a\be}):=\big(S(\nu_\a,\nu_\be)\big)=\big(\lan A^\a,A^\be\ran\big)
\end{equation}
be the {\it fundamental matrix} of $M$ w.r.t. $\nu_1,\cdots,\nu_m$ at the considered point,
whose eigenvalues are
\begin{equation}
\la_1\geq\la_2\geq\cdots \geq\la_m,
\end{equation}
depending not on the choices of the orthonormal basis of the normal space.
Then
\begin{equation}
|B|^2=\text{tr }S=\sum_{\a=1}^m \la_\a.
\end{equation}

For any positive number $p$, we consider the $C^\infty$-function
\begin{equation}
f_p:=\text{tr}(S^p)
\end{equation}
as in \cite{Lu}. Let $x\in M$ be an arbitrary point, choose a local orthonormal frame field,
such that the fundamental matrix at $x$ is diagonalized, i.e.
\begin{equation}
S_{\a\be}(x)=\la_\a\de_{\a\be}
\end{equation}
and
\begin{equation}
\la_1=\cdots=\la_r>\la_{r+1}\geq \cdots\geq \la_m.
\end{equation}
i.e. $r$ is the multiplicities of the greatest eigenvalue of $S$ at $x$.

By (\ref{LaB1}), for each $1\leq \a\leq m$,
\begin{equation}\label{de5}\aligned
\f{1}{2}\n^2 S_{\a\a}:=&\f{1}{2}(\n^2_{E_kE_k}S)(\nu_\a,\nu_\a)=h_{ijk}^\a h_{ijk}^\a+h_{ijkk}^\a h_{ij}^\a\\
=&|\n A^\a|^2+nS_{\a\a}-\sum_\be S_{\a\be}^2-\sum_\be \|[A^\a,A^\be]\|^2\\
=&|\n A^\a|^2+n\la_\a-\la_\a^2-\sum_\be \|[A^\a,A^\be]\|^2
\endaligned
\end{equation}
at $x$.
By Theorem \ref{ineq1}, if $r\leq n$, then
\begin{equation}\label{Lu2}\aligned
&\sum_{\be}\|[A^\a,A^\be]\|^2\leq \|A^\a\|^2\left(\sum_{\be\leq n,\be\neq \a}\|A^\be\|^2+\|A^2\|^2 \right)\\
=&\la_\a\left(\sum_{\be\leq n,\be\neq \a}\la_\be+\la_2\right)=\la_\a\left(\sum_{\be\leq n}\la_\be-\la_\a+\la_2\right)
\endaligned
\end{equation}
for each $1\leq \a\leq r$,
\begin{equation}\aligned
&\sum_{\be}\|[A^\a,A^\be]\|^2\leq \|A^\a\|^2\left(\sum_{\be\leq n,\be\neq \a}\|A^\be\|^2+\|A^1\|^2 \right)\\
=&\la_\a\left(\sum_{\be\leq n,\be\neq \a}\la_\be+\la_1\right)=\la_\a\left(\sum_{\be\leq n}\la_\be-\la_\a+\la_1\right)
\endaligned
\end{equation}
whenever $r+1\leq \a\leq n$, and
\begin{equation}\label{Lu3}\aligned
&\sum_{\be}\|[A^\a,A^\be]\|^2\leq \|A^\a\|^2\left(\sum_{\be\leq n-1}\|A^\be\|^2+\|A^1\|^2 \right)\\
=&\la_\a\left(\sum_{\be\leq n-1}\la_\be+\la_1\right)\leq \la_\a\left(\sum_{\be\leq n}\la_\be-\la_\a+\la_1\right)
\endaligned
\end{equation}
whenever $\a\geq n+1$. On the other hand, if $r\geq n+1$, then
(\ref{Lu2}) holds for each $1\leq \a\leq n$,
\begin{equation}\label{Lu4}\aligned
&\sum_{\be}\|[A^\a,A^\be]\|^2\leq \|A^\a\|^2\left(\sum_{\be\leq n-1}\|A^\be\|^2+\|A^2\|^2 \right)\\
=&\la_\a\left(\sum_{\be\leq n-1}\la_\be+\la_2\right)=\la_\a\left(\sum_{\be\leq n}\la_\be-\la_\a+\la_2\right)
\endaligned
\end{equation}
for each $n+1\leq \a\leq r$ and (\ref{Lu3}) holds for each $\a\geq r+1$.

Note that
\begin{equation}
f_p=\text{tr}(S^p)=\sum_{\a_1,\cdots,\a_p}S_{\a_1\a_2}S_{\a_2\a_3}\cdots S_{\a_p\a_1},
\end{equation}
we have
\begin{equation}\label{f1}\aligned
\f{1}{2}\De f_p=&\f{1}{2}\sum_{\a_1,\cdots,\a_p}(\n^2 S_{\a_1\a_2}) S_{\a_2\a_3}\cdots S_{\a_p\a_1}+\cdots+\f{1}{2}\sum_{\a_1,\cdots,\a_p}S_{\a_1\a_2}S_{\a_2\a_3}\cdots (\n^2 S_{\a_m\a_1})\\
&+\sum_{\a_1,\cdots,\a_p}\sum_{j<k}S_{\a_1\a_2}\cdots\widehat{S_{\a_j\a_{j+1}}}\cdots\widehat{S_{\a_k\a_{k+1}}}\cdots S_{\a_p\a_1}\lan \n S_{\a_j\a_{j+1}},\n S_{\a_k\a_{k+1}}\ran\\
=&p\sum_\a \left(\f{1}{2}\n^2 S_{\a\a}\right)\cdot \la_\a^{p-1}+p\sum_{\a<\be}\sum_{s+t=p-2}|\n S_{\a\be}|^2\la_\a^s \la_\be^t+\f{p(p-1)}{2}\sum_\a |\n S_{\a\a}|^2\la_\a^{p-2}\\
\geq &npf_p-p\sum_{1\leq \a\leq r}\la_\a^p\left(\sum_{\be\leq n} \la_\be+\la_2\right)-p\sum_{r+1\leq \a\leq m}\la_\a^p\left(\sum_{\be\leq n} \la_\be+\la_1\right)\\
&+p\sum_\a |\n A^\a|^2\la_\a^{p-1}+\f{p(p-1)}{2}\sum_\a |\n S_{\a\a}|^2\la_\a^{p-2}
\endaligned
\end{equation}
by using (\ref{de5})-(\ref{Lu4}).
On the other hand, with the aid of the Cauchy inequality, we get
\begin{equation}\label{f2}\aligned
|\n f_p|^2=&p^2 \sum_i\left(\sum_\a (\n_{E_i}S_{\a\a})\la_\a^{p-1}\right)^2\\
=&p^2 \sum_i\left(\sum_\a (\n_{E_i}S_{\a\a})\la_\a^{\f{p}{2}-1}\cdot \la_\a^{\f{p}{2}}\right)^2\\
\leq & p^2 f_p \sum_{\a}|\n S_{\a\a}|^2 \la_\a^{p-2}.
\endaligned
\end{equation}

Let
\begin{equation}
g_p:=(f_p)^{\f{1}{p}}.
\end{equation}
then (\ref{f1}) and (\ref{f2}) implies
\begin{equation}\aligned
\De g_p=&\f{1}{p}(f_p)^{\f{1}{p}-1}\De f_p+\f{1}{p}\left(\f{1}{p}-1\right)(f_p)^{\f{1}{p}-2}|\n f_p|^2\\
\geq & 2g_p\Biggl[n-f_p^{-1}\sum_{1\leq \a\leq r}\la_\a^p\left(\sum_{\be\leq n} \la_\be+\la_2\right)-f_p^{-1}\sum_{r+1\leq \a\leq m}\la_\a^p\left(\sum_{\be\leq n} \la_\be+\la_1\right) \\ &+f_p^{-1}\sum_\a |\n A^\a|^2 \la_\a^{p-1}\Biggl]
\endaligned
\end{equation}
whenever $f_p(x)\neq 0$. Noting that
\begin{equation}
\lim_{p\ra \infty}f_p^{-1}\la_\a^p=\lim_{p\ra\infty}\f{\left(\f{\la_\a}{\la_1}\right)^p}{\sum\limits_\be \left(\f{\la_\be}{\la_1}\right)^p}=\left\{
\begin{matrix}\f{1}{r} & \a\leq r\\ 0 & \a\geq r+1\end{matrix}\right.
\end{equation}
we have
\begin{equation}
\lim_{p\ra \infty}\De g_p\geq 2\la_1\left(n-\sum_{\be\leq n} \la_\be-\la_2\right)+\f{2}{r}\sum_{1\leq \a\leq r}|\n A^\a|^2.
\end{equation}
In conjunction with $\int_M \De g_p=0$, we get the following Simons type integral inequality:
\begin{pro}\label{p1}
Let $M$ be an $n$-dimensional compact minimal submanifold in $S^{n+m}$, 
 $\la_1=\cdots=\la_r>\la_{r+1}\geq \cdots\geq \la_m$ be the eigenvalues of the fundamental matrix at each considered point, then
\begin{equation}
\int_M \left(\la_1\left(n-\sum_{\be=1}^{\min\{n,m\}}\la_\be-\la_2\right)+\f{1}{r}\sum_{\a=1}^r|\n A^\a|^2\right)*1\leq 0.
\end{equation}
\end{pro}
This inequality enable us to establish a pinching theorem as follows.
\begin{thm}
Let $M$ be an $n$-dimensional compact minimal submanifold in $S^{n+m}$, if for each $x\in M$,
\begin{equation}
\sum_{\a=1}^{\min\{n,m\}}\la_\a+\la_2\leq n
\end{equation}
with $\la_1\geq \cdots\geq \la_m$ be the eigenvalues of the fundamental matrix at $x$, then
$M$ is totally geodesic or $\sum\limits_{\a=1}^{\min\{n,m\}}\la_\a+\la_2\equiv n$.
\end{thm}

We call a minimal submanifold $M^n\subset S^{n+m}$ be a {\it Simons-Chern-do Carmo-Kobayashi ideal manifold (SCDK-ideal manifold)} whenever $\sum\limits_{\a=1}^{\min\{n,m\}}\la_\a+\la_2\equiv n$.
As shown in the proof of Proposition \ref{p1}, if $M$ is a compact SCDK-ideal manifold,
then the following properties hold for almost all $x\in M$:
Let $\nu_1,\cdots,\nu_m$ be an orthonormal basis of $N_x M$
satisfying
\begin{equation}
\lan A^\a,A^\be\ran=\la_\a \de_{\a\be},
\end{equation}
then
\begin{equation}\label{c1}
\sum_\be \|[A^\a,A^\be]\|^2=\la_\a\left(\sum_{\be=1}^{\min\{n,m\}}\la_\be-\la_\a+\la_2\right)\qquad \forall \a=1,\cdots, r
\end{equation}
and
\begin{equation}\label{c2}
|\n A^\a|^2=0\qquad \forall \a=1,\cdots, r.
\end{equation}
Particularly, (\ref{c1}) and (\ref{c2}) hold true whenever $r$ is a constant on a neighborhood of $x$.

\subsection{Generalized Clifford tori and Veronese manifolds}

Let $S^q(r)\subset \R^{q+1}$ denote the $q$-dimensional Euclidean sphere centered at the origin and of radius $r$,
$\phi_q:S^q(r)\ra \R^{q+1}$ be the including mapping, then for any $1\leq p\leq q$, $\phi:S^p\left(\sqrt{\f{p}{p+q}}\right)\times S^q\left(\sqrt{\f{q}{p+q}}\right)\ra \R^{p+q+2}$
\begin{equation}
\phi(x,y)=(\phi_p(x),\phi_q(y))
\end{equation}
is an isometric embedding, whose image lies in $S^{p+q+1}$. We call $M_{p,q}:=S^p\left(\sqrt{\f{p}{p+q}}\right)\times S^q\left(\sqrt{\f{q}{p+q}}\right)$ a {\it generalized Clifford torus} (see e.g \cite{C-D-K}).
For each $(x,y)\in M_{p,q}$, we have
\begin{equation}\aligned
T_{(x,y)}M_{p,q}=&T_x S^p\left(\sqrt{\f{p}{p+q}}\right)\oplus T_y S^q\left(\sqrt{\f{q}{p+q}}\right),\\
T_{(x,y)}S^{p+q+1}=&T_{(x,y)}M_{p,q}\oplus N_{(x,y)}M_{p,q},\\
\R^{p+q+2}=&T_{(x,y)}S^{p+q+1}\oplus \R\phi(x,y).
\endaligned
\end{equation}
Due to the geometric properties of Euclidean spheres, for any $X_1,X_2\in T_x S^p\left(\sqrt{\f{p}{p+q}}\right)$,
\begin{equation}
\n_{X_1}(d\phi)(X_2)=(\n_{X_1}(d\phi_p)(X_2),0)=-\f{p+q}{p}\lan X_1,X_2\ran(\phi_p(x),0)
\end{equation}
and hence
\begin{equation}
B_{X_1X_2}=\n_{X_1}(d\phi)(X_2)^\bot=\lan X_1,X_2\ran\left(-\f{q}{p}\phi_p(x),\phi_q(y)\right)
\end{equation}
with $(\cdots)^\bot$ the orthogonal projection from $\R^{p+q+2}$ onto $T_{(x,y)}S^{p+q+1}$. Similarly
\begin{equation}
B_{Y_1Y_2}=\lan Y_1,Y_2\ran\left(\phi_p(x),-\f{p}{q}\phi_q(y)\right)
\end{equation}
for any $Y_1,Y_2\in T_y S^q\left(\sqrt{\f{q}{p+q}}\right)$ and
\begin{equation}
B_{X_1Y_1}=0.
\end{equation}
Let $e_1,\cdots,e_p$ be an orthonormal basis of $T_x S^p\left(\sqrt{\f{p}{p+q}}\right)$ and $e_{p+1},\cdots,e_{p+q}$
be an orthonormal basis of $T_y S^q\left(\sqrt{\f{q}{p+q}}\right)$, then
\begin{equation}
H=\sum_{i=1}^{p+q}B_{e_ie_i}=0
\end{equation}
and
\begin{equation}
|B|^2=\sum_{i=1}^{p+q}|B_{e_ie_j}|^2=p+q.
\end{equation}
In a summary, $M_{p,q}$ is an $n$-dimensional ($n=p+q$) minimal submanifold in the Euclidean sphere, satisfying
$\la_1\equiv n$ and $\la_\a\equiv 0$ for each $\a\geq 2$.

Let $\text{Sym}_{n+1}^0$ be the vector space of all traceless symmetric $(n+1)\times (n+1)$-matrices,
which has dimension $\f{n(n+3)}{2}$, then $\psi:S^n\Big(\sqrt{\f{2(n+1)}{n}}\Big)\ra \text{Sym}_{n+1}^0$
\begin{equation}
\psi(x)=\f{1}{2}\sqrt{\f{n}{n+1}}\left(xx^T-\f{2}{n}I_{n+1}\right)
\end{equation}
is an isometric embedding, whose image lies in the unit hypersphere of $\text{Sym}_{n+1}^0$.
(Here $x$ is seen as a column vector.) This is called a {\it Veronese $n$-manifold} (see e.g \cite{C-D-K}).
For any $x,y\in M:=S^n\Big(\sqrt{\f{2(n+1)}{n}}\Big)$, let $U$ be an orthogonal matrix
satisfying $Ux=y$, then
$$S\in \text{Sym}_{n+1}^0\mapsto USU^T\in \text{Sym}_{n+1}^0$$
is an isometric transformation on $\text{Sym}_{n+1}^0$ which maps $\psi(x)$ to $\psi(y)$.
This means each Veronese manifold is homogenous. Let $\ep_1,\cdots,\ep_{n+1}$ be the canonical basis
of $\R^{n+1}$, then at $x_0:=\sqrt{\f{2(n+1)}{n}}\ep_{n+1}\in M$,
\begin{equation}
\psi_* v=\f{1}{2}\sqrt{\f{n}{n+1}}(x_0v^T+vx_0^T)=\f{1}{\sqrt{2}}(\ep_{n+1}v^T+v\ep_{n+1}^T)
\end{equation}
for any $v\in T_{x_0}M$, and hence
\begin{equation}
\lan \psi_*\ep_i,\psi_*\ep_j\ran=\left\lan \f{1}{\sqrt{2}}(E_{n+1,i}+E_{i,n+1}),\f{1}{\sqrt{2}}(E_{n+1,j}+E_{j,n+1})\right\ran=\de_{ij},
\end{equation}
i.e. $\ep_1,\cdots,\ep_n$ is an orthonormal basis of $T_{x_0}M$. For any $\nu\in N_{x_0}M$,
\begin{equation}\label{Veronese1}
\lan B_{\ep_i\ep_j},\nu\ran=\f{1}{2}\sqrt{\f{n}{n+1}}\lan \ep_i\ep_j^T+\ep_j\ep_i^T,\nu\ran=\f{1}{2}\sqrt{\f{n}{n+1}}\lan E_{ij}+E_{ji},\nu\ran.
\end{equation}
Denote
\begin{equation}\label{Veronese2}\aligned
&\nu_{11}:=\f{1}{\sqrt{2}}(-E_{11}+E_{22}),\quad \nu_{22}:=\f{1}{\sqrt{6}}(-(E_{11}+E_{22})+2E_{33})=\f{1}{\sqrt{6}}(-I_2+2E_{33}),\cdots\\
&\nu_{jj}:=\f{1}{\sqrt{j(j+1)}}(-I_j+jE_{j+1,j+1}),\cdots,\nu_{n-1,n-1}:=\f{1}{\sqrt{(n-1)n}}(-I_{n-1}+(n-1)E_{nn})
\endaligned
\end{equation}
and
\begin{equation}\label{Veronese3}
\nu_{ij}:=\f{1}{\sqrt{2}}(E_{ij}+E_{ji})\qquad \forall 1\leq i<j\leq n,
\end{equation}
then it is easy to check that all such vectors compose an orthonormal basis of $N_{x_0}M$.
Substituting (\ref{Veronese2}) and (\ref{Veronese3}) into (\ref{Veronese1}) gives
\begin{equation}
\lan B_{\ep_i\ep_j},\nu_{jj}\ran=\begin{cases}
-\sqrt{\f{n}{n+1}}\cdot \f{1}{\sqrt{j(j+1)}} & \text{if }i\leq j\\
\sqrt{\f{n}{n+1}}\cdot\sqrt{\f{j}{j+1}}     & \text{if }i=j+1\\
0 & \text{if }i\geq j+2
\end{cases}
\end{equation}
\begin{equation}
\lan B_{\ep_i\ep_j},\nu_{ij}\ran=\lan B_{\ep_j\ep_i},\nu_{ij}\ran=\sqrt{\f{n}{2(n+1)}}\qquad \forall 1\leq i<j\leq n
\end{equation}	
and $\lan B_{\ep_i\ep_j},\nu_{kl}\ran=0$ for other cases. In other words,
\begin{equation}
A^{\nu_{jj}}=\sqrt{\f{n}{n+1}}\cdot \f{1}{\sqrt{j(j+1)}}(-I_j+jE_{j+1,j+1})\qquad \forall 1\leq j\leq n-1
\end{equation}
and
\begin{equation}
A^{\nu_{ij}}=\sqrt{\f{n}{2(n+1)}}(E_{ij}+E_{ji})\qquad \forall 1\leq i<j\leq n.
\end{equation}
It follows that $M$ is a minimal submanifold with
\begin{equation}
\la_1\equiv \la_2\equiv \cdots\equiv \la_{\f{(n-1)(n+2)}{2}}\equiv \f{n}{n+1}
\end{equation}
and $\la_\a\equiv 0$ for each $\a>\f{(n-1)(n+2)}{2}$.

In a summary, we have the following conclusion:
\begin{thm}\label{Clifford-Veronese}
Both generalized Clifford tori and Veronese submanifolds of all dimensions are compact minimal SCDK-ideal manifolds in Euclidean spheres.
\end{thm}

	\bigskip\bigskip
	\Section{On compact SCDK-ideal 3-manifolds}{On compact SCDK-ideal 3-manifolds}

Now we assume $M$ is a 3-dimensional compact SCDK-ideal manifold in $S^{3+m}$.
Let $\la_1\geq \cdots\geq \la_m$ be the eigenvalues of the fundamental matrix at each point,
then $M=\Om_1\cup \Om_2 \cup \Om_3$ with
\begin{equation}
\aligned
\Om_1:=&\{x\in M: \la_3>\la_4\},\\
\Om_2:=&\{x\in M: \la_2>\la_3\},\\
\Om_3:=&\{x\in M:\la_2=\la_3=\la_4\}.
\endaligned
\end{equation}
We shall consider these domains case by case.

\subsection{Case I: $\la_3>\la_4$}

Obviously $\Om_1$ is an open subset of $M$. If there exists $x\in \Om_1$ such that $\la_1>\la_2$,
then $r=1$ on a neighborhood $U_x$ of $x$. Thereby, in conjunction with (\ref{c1})
and Proposition \ref{dim3}, we have $\la_2=\la_3$, and there exist an orthonormal tangent frame
field $E_1,E_2,E_3$ and normal fields $\nu_1,\nu_2,\nu_3$ on $U_x$, such that:
\begin{itemize}
\item $\lan \nu_\a,\nu_\be\ran=\de_{\a\be}$ and $\lan A^\a,A^\be\ran=\la_\a \de_{\a\be}$ for $1\leq \a,\be\leq 3$;
\item $A^1=\f{\mu_1}{\sqrt{6}}\text{diag}(2,-1,-1)$, $A^2=\f{\mu_2}{\sqrt{2}}(E_{12}+E_{21})$ and $A^3=\f{\mu_2}{\sqrt{2}}(E_{13}+E_{31})$
with $A^\a:=(h_{ij}^\a)$, $\mu_\a:=\sqrt{\la_\a}$;
\item $A^\nu(E_1)=0$ for each normal vector $\nu$ orthogonal to all of $\nu_1,\nu_2,\nu_3$.
\end{itemize}
Therefore
\begin{equation}\label{B1}\aligned
B_{E_1E_1}=\f{2\mu_1}{\sqrt{6}}\nu_1,\ &B_{E_1E_2}=\f{\mu_2}{\sqrt{2}}\nu_2,\ B_{E_1E_3}=\f{\mu_2}{\sqrt{2}}\nu_3,\\
B_{E_2E_2}=-\f{\mu_1}{\sqrt{6}}\nu_1+B_{E_2E_2}^\bot,\ &B_{E_3E_3}=-\f{\mu_1}{\sqrt{6}}\nu_1+B_{E_3E_3}^\bot,\ B_{E_2E_3}=B_{E_2E_3}^\bot.
\endaligned
\end{equation}
Here $(\cdots)^\bot$ denotes the projection onto the orthogonal complement of the subspace spanned by $\nu_1,\nu_2,\nu_3$.

(\ref{c2}) implies
\begin{equation}
h_{ijk}^1=0\qquad \forall 1\leq i,j,k\leq 3.
\end{equation}
Then for each $j$,
\begin{equation}
0=h_{11j}^1=\lan (\n_{E_j}B)_{E_1E_1},\nu_1\ran=\f{2}{\sqrt{6}}\n_{E_j}(\mu_1)
\end{equation}
forces $\mu_1\equiv \text{const}$ on $U_x$. Combining with
\begin{equation}
3=\sum_{\a=1}^3 \la_\a+\la_2=\mu_1^2+3\mu_2^2,
\end{equation}
we see $\mu_2$ is locally constant. Next, a direct calculation based on (\ref{B1}) gives
\begin{equation}\label{dB1}
\aligned
(\n_{E_j}B)_{E_1E_1}=&\f{2\mu_1}{\sqrt{6}}\n_{E_j}\nu_1-\sqrt{2}\mu_2\lan \n_{E_j}E_1,E_2\ran \nu_2-\sqrt{2}\mu_2\lan \n_{E_j}E_1,E_3\ran \nu_3,\\
(\n_{E_j}B)_{E_1E_2}=&\f{\mu_2}{\sqrt{2}}\n_{E_j}\nu_2+\f{3\mu_1}{\sqrt{6}}\lan \n_{E_j}E_1,E_2\ran \nu_1-\f{\mu_2}{\sqrt{2}}\lan \n_{E_j}E_2,E_3\ran \nu_3\\
&-\lan \n_{E_j}E_1,E_2\ran B_{E_2E_2}^\bot-\lan \n_{E_j}E_1,E_3\ran B_{E_2E_3}^\bot,\\
(\n_{E_j}B)_{E_1E_3}=&\f{\mu_2}{\sqrt{2}}\n_{E_j}\nu_3+\f{3\mu_1}{\sqrt{6}}\lan \n_{E_j}E_1,E_3\ran \nu_1+\f{\mu_2}{\sqrt{2}}\lan \n_{E_j}E_2,E_3\ran \nu_2\\
&-\lan \n_{E_j}E_1,E_3\ran B_{E_3E_3}^\bot-\lan \n_{E_j}E_1,E_2\ran B_{E_2E_3}^\bot,\\
(\n_{E_j}B)_{E_2E_2}=&\n_{E_j}B_{E_2E_2}+\sqrt{2}\mu_2\lan \n_{E_j}E_1,E_2\ran \nu_2-2\lan \n_{E_j}E_2,E_3\ran B_{E_2E_3}^\bot,\\
(\n_{E_j}B)_{E_3E_3}=&\n_{E_j}B_{E_3E_3}+\sqrt{2}\mu_2\lan \n_{E_j}E_1,E_3\ran \nu_3+2\lan \n_{E_j}E_2,E_3\ran B_{E_2E_3}^\bot.
\endaligned
\end{equation}
By $h_{12j}^1=h_{13j}^1=0$, we have
\begin{equation}\label{CaseI.1}
\lan \n_{E_j}\nu_1,\nu_2\ran=\f{\sqrt{3}\mu_1}{\mu_2}\lan \n_{E_j}E_1,E_2\ran
\end{equation}
and
\begin{equation}\label{CaseI.2}
\lan \n_{E_j}\nu_1,\nu_3\ran=\f{\sqrt{3}\mu_1}{\mu_2}\lan \n_{E_j}E_1,E_3\ran.
\end{equation}
Combining $h_{121}^2=h_{112}^2$ (a corollary of Codazzi equations) and (\ref{CaseI.1}) yields
$$0=\f{2\mu_1}{\sqrt{6}}\lan \n_{E_2}\nu_1,\nu_2\ran-\sqrt{2}\mu_2\lan \n_{E_2}E_1,E_2\ran=\f{\sqrt{2}(\mu_1^2-\mu_2^2)}{\mu_2}\lan \n_{E_2}E_1,E_2\ran,$$
i.e.
\begin{equation}\label{CaseI.10}
\lan \n_{E_2}E_1,E_2\ran=0.
\end{equation}
Similarly, in conjunction with $h_{131}^3=h_{113}^3$ and (\ref{CaseI.2}, we get
\begin{equation}\label{CaseI.18}
\lan \n_{E_3}E_1,E_3\ran=0.
\end{equation}
Due to $h_{123}^2=h_{132}^2$, we have
$$0=\f{\mu_2}{\sqrt{2}}\lan \n_{E_2}\nu_3,\nu_2\ran+\f{\mu_2}{\sqrt{2}}\lan \n_{E_2}E_2,E_3\ran,$$
i.e.
\begin{equation}
\lan \n_{E_2}\nu_2,\nu_3\ran=\lan \n_{E_2}E_2,E_3\ran.
\end{equation}
Similarly, $h_{132}^3=h_{123}^3$ immediately gives
\begin{equation}\label{CaseI.3}
\lan \n_{E_3}\nu_2,\nu_3\ran=\lan \n_{E_3}E_2,E_3\ran.
\end{equation}
$h_{122}^2=h_{221}^2$ means
$$0=\lan \n_{E_1}B_{E_2E_2},\nu_2\ran+\sqrt{2}\mu_2\lan \n_{E_1}E_1,E_2\ran$$
i.e.
\begin{equation}\label{CaseI.4}
\lan \n_{E_1}\nu_2,B_{E_2E_2}\ran=\sqrt{2}\mu_2\lan \n_{E_1}E_1,E_2\ran.
\end{equation}
On the other hand, combining $h_{331}^2=h_{133}^2$ and (\ref{CaseI.3}) gives
$$\lan \n_{E_1}B_{E_3E_3},\nu_2\ran=\f{\mu_2}{\sqrt{2}}\lan \n_{E_3}\nu_3,\nu_2\ran+\f{\mu_2}{\sqrt{2}}\lan \n_{E_3}E_2,E_3\ran=0$$
i.e.
\begin{equation}\label{CaseI.5}
\lan \n_{E_1}\nu_2,B_{E_3E_3}\ran=0.
\end{equation}
Adding both hand sides of (\ref{CaseI.4}) and (\ref{CaseI.5}), we get
$$\aligned
\sqrt{2}\mu_2\lan \n_{E_1}E_1,E_2\ran=&\lan \n_{E_1}\nu_2,B_{E_2E_2}+B_{E_3E_3}\ran=-\lan \n_{E_1}\nu_2,B_{E_1E_1}\ran\\
=&\f{2\mu_1}{\sqrt{6}}\lan \n_{E_1}\nu_1,\nu_2\ran=\f{\sqrt{2}\mu_1^2}{\mu_2}\lan \n_{E_1}E_1,E_2\ran.
\endaligned$$
Thus
\begin{equation}\label{CaseI.6}
\lan \n_{E_1}E_1,E_2\ran=0.
\end{equation}
Similarly, due to $h_{133}^3=h_{331}^3$ and $h_{221}^3=h_{122}^3$, we can derive
\begin{equation}\label{CaseI.7}
\lan \n_{E_1}E_1,E_3\ran=0.
\end{equation}
It follows from $h_{121}^3=h_{112}^3$ that
\begin{equation*}\aligned
\frac{\mu_2}{\sqrt{2}}\big(\lan\n_{E_1}\nu_2,\nu_3\ran-\lan \n_{E_1}E_2,E_3\ran\big)=&\f{2\mu_1}{\sqrt{6}}\lan \n_{E_2}\nu_1,\nu_3\ran-\sqrt{2}\mu_2\lan \n_{E_2}E_1,E_3\ran\\
=&\f{\sqrt{2}(\mu_1^2-\mu_2^2)}{\mu_2}\lan \n_{E_2}E_1,E_3\ran,
\endaligned
\end{equation*}
i.e.
\begin{equation}\label{CaseI.8}
\lan\n_{E_1}\nu_2,\nu_3\ran-\lan \n_{E_1}E_2,E_3\ran=\f{2(\mu_1^2-\mu_2^2)}{\mu_2^2}\lan \n_{E_2}E_1,E_3\ran.
\end{equation}
A similar calculation based on $h_{131}^2=h_{113}^2$ gives
\begin{equation}
\lan\n_{E_1}\nu_3,\nu_2\ran-\lan \n_{E_1}E_3,E_2\ran=\f{2(\mu_1^2-\mu_2^2)}{\mu_2^2}\lan \n_{E_3}E_1,E_2\ran.
\end{equation}
Therefore
\begin{equation}\label{CaseI.11}
\lan \n_{E_3}E_1,E_2\ran=-\lan \n_{E_2}E_1,E_3\ran.
\end{equation}

In conjunction with (\ref{CaseI.6}) and (\ref{CaseI.7}), we get
$$\f{\mu_2}{\sqrt{2}}(\n_{E_1}{\nu_2})^\bot=(\n_{E_1}B)_{E_1E_2}^\bot=(\n_{E_2}B)_{E_1E_1}^\bot=\f{2\mu_1}{\sqrt{6}}(\n_{E_2}\nu_1)^\bot$$
by again using (\ref{dB1}), i.e.
\begin{equation}\label{CaseI.9}
(\n_{E_1}{\nu_2})^\bot=\f{2\mu_1}{\sqrt{3}\mu_2}(\n_{E_2}{\nu_1})^\bot.
\end{equation}
Thus
\begin{equation}\label{CaseI.12}\aligned
&\lan \n_{E_2}\nu_1,\n_{E_1}\nu_2\ran=\lan \n_{E_2}\nu_1,\nu_3\ran+\lan \n_{E_1}\nu_2,\nu_3\ran+\lan (\n_{E_2}\nu_1)^\bot,(\n_{E_1}\nu_2)^\bot\ran\\
=&\f{\sqrt{3}\mu_1}{\mu_2}\lan \n_{E_2}E_1,E_3\ran\left(\lan \n_{E_1}E_2,E_3\ran+\f{2(\mu_1^2-\mu_2^2)}{\mu_2^2}\lan \n_{E_2}E_1,E_3\ran\right)+\lan (\n_{E_2}\nu_1)^\bot,(\n_{E_1}\nu_2)^\bot\ran\\
=&\f{\sqrt{3}\mu_1}{\mu_2}\lan \n_{E_2}E_1,\n_{E_1}E_2\ran+\f{2\sqrt{3}\mu_1(\mu_1^2-\mu_2^2)}{\mu_2^3}\lan \n_{E_2}E_1,E_3\ran^2+\f{2\mu_1}{\sqrt{3}\mu_2}|(\n_{E_2}\nu_1)^\bot|^2.
\endaligned
\end{equation}
Here we have used (\ref{CaseI.2}), (\ref{CaseI.8}) and (\ref{CaseI.9}). A similar calculation as above shows
\begin{equation}\label{CaseI.13}
\lan \n_{E_3}\nu_1,\n_{E_1}\nu_3\ran=\f{\sqrt{3}\mu_1}{\mu_2}\lan \n_{E_3}E_1,\n_{E_1}E_3\ran+\f{2\sqrt{3}\mu_1(\mu_1^2-\mu_2^2)}{\mu_2^3}\lan \n_{E_3}E_1,E_2\ran^2+\f{2\mu_1}{\sqrt{3}\mu_2}|(\n_{E_3}\nu_1)^\bot|^2
\end{equation}
On the other hand, substituting (\ref{CaseI.10}) into (\ref{dB1}) gives
\begin{equation}
\aligned
\f{\mu_2}{\sqrt{2}}(\n_{E_2}\nu_2)^\bot=&(\n_{E_2}B)_{E_1E_2}^\bot+\lan \n_{E_2}E_1,E_3\ran B_{E_2E_3}=(\n_{E_1}B)_{E_2E_2}^\bot+\lan \n_{E_2}E_1,E_3\ran B_{E_2E_3}\\
=&(\n_{E_1}B_{E_2E_2})^\bot+\Big(\lan \n_{E_2}E_1,E_3\ran-2\lan \n_{E_1}E_2,E_3\ran\Big) B_{E_2E_3}
\endaligned
\end{equation}
and similarly
\begin{equation}
\f{\mu_2}{\sqrt{2}}(\n_{E_3}\nu_3)^\bot=(\n_{E_1}B_{E_3E_3})^\bot+\Big(\lan \n_{E_3}E_1,E_2\ran+2\lan \n_{E_1}E_2,E_3\ran\Big) B_{E_2E_3}.
\end{equation}
Adding both hand sides of the above equations and then using (\ref{CaseI.11}), we obtain
\begin{equation}\label{CaseI.14}
(\n_{E_2}\nu_2)^\bot+(\n_{E_3}\nu_3)^\bot=\f{\sqrt{2}}{\mu_2}(\n_{E_1}B_{E_1E_1})^\bot=\f{2\mu_1}{\sqrt{3}\mu_2}(\n_{E_1}\nu_1)^\bot.
\end{equation}
Therefore
\begin{equation}
\aligned
&\lan \n_{E_1}\nu_1,\n_{E_2}\nu_2\ran+\lan \n_{E_1}\nu_1,\n_{E_3}\nu_3\ran\\
=&\lan \n_{E_1}\nu_1,\nu_3\ran\lan \n_{E_2}\nu_2,\nu_3\ran+\lan \n_{E_1}\nu_1,\nu_2\ran\lan \n_{E_3}\nu_3,\nu_2\ran+\lan (\n_{E_1}\nu_1)^\bot,(\n_{E_2}\nu_2)^\bot+(\n_{E_3}\nu_3)^\bot\ran\\
=&-\f{2\mu_1}{\sqrt{3}\mu_2}|(\n_{E_1}\nu_1)^\bot|^2.
\endaligned
\end{equation}
(Note that here $\lan \n_{E_1}\nu_1,\nu_3\ran=0$ and $\lan \n_{E_1}\nu_1,\nu_2\ran=0$ are direct corollaries of (\ref{CaseI.1}), (\ref{CaseI.2}),
(\ref{CaseI.6}) and (\ref{CaseI.7}).)

Let
\begin{equation}
R_{ijkl}:=\lan R_{E_iE_j}E_k,E_l\ran
\end{equation}
and
\begin{equation}
R_{ij\a\be}^N:=\lan R_{E_iE_j}\nu_\a,\nu_\be\ran.
\end{equation}
Then
\begin{equation}\label{CaseI.15}\aligned
\sum_{i=2}^3 R_{1i1i}^N
=&\sum_{i=2}^3 \lan -\n_{E_1}\n_{E_i}\nu_1+\n_{E_i}\n_{E_1}\nu_1+\n_{[E_1,E_i]}\nu_1,\nu_i\ran\\
=&\sum_{i=2}^3 \Big(- \n_{E_1}\lan \n_{E_i}\nu_1,\nu_i\ran+\n_{E_i}\lan \n_{E_1}\nu_1,\nu_i\ran+\n_{[E_1,E_i]}\nu_1,\nu_i\ran\Big)\\
&+\sum_{i=2}^3 \Big(\lan \n_{E_i}\nu_1,\n_{E_1}\nu_i\ran-\lan \n_{E_1}\nu_1,\n_{E_i}\nu_i\ran\Big)\\
=&\f{\sqrt{3}\mu_1}{\mu_2}\sum_{i=2}^3 R_{1i1i}+\sum_{i=2}^3 \Big(-\f{\sqrt{3}\mu_1}{\mu_2}\lan \n_{E_i}E_1,\n_{E_1}E_i\ran+\f{\sqrt{3}\mu_1}{\mu_2}\lan \n_{E_1}E_1,\n_{E_i}E_i\ran\\
&+\lan \n_{E_i}\nu_1,\n_{E_1}\nu_i\ran-\lan \n_{E_1}\nu_1,\n_{E_i}\nu_i\ran\Big)\\
=&\f{\sqrt{3}\mu_1}{\mu_2}\sum_{i=2}^3 R_{1i1i}+\f{4\sqrt{3}\mu_1(\mu_1^2-\mu_2^2)}{\mu_2^3}\lan \n_{E_2}E_1,E_3\ran^2+\f{2\mu_1}{\sqrt{3}\mu_2}\sum_{j=1}^3 |(\n_{E_j}\nu_1)^\bot|^2.
\endaligned
\end{equation}
Here we have used (\ref{CaseI.1}), (\ref{CaseI.2}), (\ref{CaseI.6}), (\ref{CaseI.7}), (\ref{CaseI.11}), (\ref{CaseI.12}), (\ref{CaseI.13}) and (\ref{CaseI.14}).
According to Gauss equations and Ricci equations, we have
\begin{equation}\label{CaseI.16}\aligned
R_{1i1i}=&\lan \ol{R}_{E_1E_i}E_1,E_i\ran+\lan B_{E_1E_1},B_{E_iE_i}\ran-\lan B_{E_1E_i},B_{E_1E_i}\ran\\
=&1-\f{\mu_1^2}{3}-\f{\mu_2^2}{2}=\f{\mu_1^2+3\mu_2^2}{3}-\f{\mu_1^2}{3}-\f{\mu_2^2}{2}=\f{\mu_2^2}{2}
\endaligned
\end{equation}
and
\begin{equation}\label{CaseI.17}
R_{1i1i}^N=\lan \ol{R}_{E_1E_i}\nu_1,\nu_i\ran+\lan A^1(E_1),A^i(E_i)\ran-\lan A^1(E_i),A^i(E_1)\ran=\f{\mu_1\mu_2}{2\sqrt{3}}
\end{equation}
for $i=2$ or $3$. Substituting (\ref{CaseI.16}) and (\ref{CaseI.17}) into (\ref{CaseI.15}), we have
\begin{equation}\label{CaseI.19}
\lan \n_{E_2}E_1,E_3\ran=0.
\end{equation}
Combining (\ref{CaseI.10}), (\ref{CaseI.18}), (\ref{CaseI.6}), (\ref{CaseI.7}), (\ref{CaseI.11})
and (\ref{CaseI.19}) implies $\n E_1=0$ and hence $R_{1212}=0$, which caused a contradiction to (\ref{CaseI.16}).

Thereby, $\la_1=\la_2$ holds everywhere on $\Om_1$.
For each $x\in \Om_1$ such that (\ref{c1}) and (\ref{c2}) hold, by Proposition \ref{dim3}, we have $\la_2=\la_3$ and there exist an orthonormal basis $e_1,e_2,e_3$ of $T_x M$
and an orhonormal basis $\nu_1,\cdots,\nu_m$ of $N_x M$, such that
\begin{equation}\aligned
A^1=\f{\mu_1}{\sqrt{6}}\text{diag}(2,-1,-1),
\quad&
A^2=\f{\mu_2}{\sqrt{2}}(E_{12}+E_{21})\\
A^3=\f{\mu_2}{\sqrt{2}}(E_{13}+E_{31}),\quad&
A^4=\f{\mu_4}{\sqrt{2}}(E_{23}+E_{32}),\\
A^5=\f{\mu_5}{\sqrt{2}}\text{diag}(0,-1,1),\quad&
A^\a=0\ (\forall \a\geq 6).
\endaligned
\end{equation}
A direct calculation shows
\begin{equation}\aligned
&[A^2,A^1]=-\f{\sqrt{3}\mu^2}{2}(E_{12}-E_{21}),\quad [A^2,A^3]=\f{\mu^2}{2}(E_{23}-E_{32}),\\
&[A^2,A^4]=\f{\mu\mu_4}{2}(E_{13}-E_{31}),\quad [A^2,A^5]=-\f{\mu\mu_5}{2}(E_{12}-E_{21})
\endaligned
\end{equation}
with $\mu:=\mu_1=\mu_2$ and hence
\begin{equation}
\sum_\be \|[A^2,A^\be]\|^2=2\la^2+\f{\la}{2}(\la_4+\la_5)<3\la^2=\la_2\left(\sum_{\be\leq 3}\la_\be-\la_2+\la_2\right)
\end{equation}
with $\la:=\la_1=\la_2=\la_3$, causing a contradiction to (\ref{c1}).

Therefore, $\Om_1$ has to be an empty set. Equivalently, we arrive at the following conclusion:
\begin{lem}\label{CaseI}
Let $M$ be a 3-dimensional compact SCDK-ideal manifold in $S^{3+m}$, then $\la_3=\la_4$ holds everywhere on $M$.
\end{lem}

\subsection{Case II: $\la_2>\la_3$}

Denote
\begin{equation}
\aligned
D_1:=&\{x\in \Om_2:\la_1>\la_2\},\\
D_2:=&\{x\in D_1:\la_3>0\},\\
D_3:=&\{x\in D_2:\la_5>0\}.
\endaligned
\end{equation}
Then $\Om_2\supset D_1\supset D_2\supset D_3$ and all of them are open subsets of $M$.
For any $x\in D_3$, (\ref{c1}) and Proposition \ref{dim3} ensures the existences of
an orthonormal tangent frame field $E_1,E_2,E_3$ and an orthonormal normal frame field
$\nu_1,\cdots,\nu_m$ on a neighborhood $U_x$ of $x$, such that
$A^1=\f{\mu_1}{\sqrt{2}}\mathrm{diag}(1,-1,0)$, $A^2=\f{\mu_2}{\sqrt{2}}(E_{12}+E_{21})$,
$A^3=\f{\mu_3}{\sqrt{2}}(E_{13}+E_{31})$, $A^4=\f{\mu_3}{\sqrt{2}}(E_{23}+E_{32})$,
$A^5=\f{\mu_5}{\sqrt{6}}\mathrm{diag}(-1,-1,2)$ and $A^\a=0$ for each $\a\geq 6$.
In other words,
\begin{equation}\label{B2}\aligned
B_{E_1E_1}=\f{\mu_1}{\sqrt{2}}\nu_1-\f{\mu_5}{\sqrt{6}}\nu_5,\ &B_{E_1E_2}=\f{\mu_2}{\sqrt{2}}\nu_2,\ B_{E_1E_3}=\f{\mu_3}{\sqrt{2}}\nu_3,\\
B_{E_2E_2}=-\f{\mu_1}{\sqrt{2}}\nu_1-\f{\mu_5}{\sqrt{6}}\nu_5,\ &B_{E_2E_3}=\f{\mu_3}{\sqrt{2}}\nu_4,\ B_{E_3E_3}=\f{2\mu_5}{\sqrt{6}}\nu_5.
\endaligned
\end{equation}
On the other hand, (\ref{c2}) means
\begin{equation}
h_{ijk}^1=0\qquad \forall 1\leq i,j,k\leq 3.
\end{equation}

In conjunction with
\begin{equation}
0=h_{11j}^1=\lan (\n_{E_j}B)_{E_1E_1},\nu_1\ran=\f{1}{\sqrt{2}}\n_{E_j}(\mu_1)+\f{\mu_5}{\sqrt{6}}\lan \n_{E_j}\nu_1,\nu_5\ran
\end{equation}
and
\begin{equation}
0=h_{22j}^1=\lan (\n_{E_j}B)_{E_2E_2},\nu_1\ran=-\f{1}{\sqrt{2}}\n_{E_j}(\mu_1)+\f{\mu_5}{\sqrt{6}}\lan \n_{E_j}\nu_1,\nu_5\ran,
\end{equation}
we have
\begin{equation}\label{CaseII.32}
\lan \n_{E_j}\nu_1,\nu_5\ran=0
\end{equation}
and
\begin{equation}
\mu_1\equiv \text{const}.
\end{equation}
Combining with
\begin{equation}\label{CaseII.2}
3=\sum_{\a=1}^3 \la_\a+\la_2=\la_1+2\la_2+\la_3,
\end{equation}
we get
\begin{equation}\label{CaseII.10}
2 \n_{E_j}\la_2+\n_{E_j}\la_3=0.
\end{equation}
Then a straightforward calculation based on (\ref{B2}) shows
\begin{equation}\label{dB2}\aligned
(\n_{E_j}B)_{E_1E_1}=&\f{\mu_1}{\sqrt{2}}\n_{E_j}\nu_1-\f{1}{\sqrt{6}}\n_{E_j}(\mu_5\nu_5)-\sqrt{2}\mu_2\lan \n_{E_j} E_1,E_2\ran \nu_2-\sqrt{2}\mu_3 \lan \n_{E_j}E_1,E_3\ran \nu_3,\\
(\n_{E_j}B)_{E_2E_2}=&-\f{\mu_1}{\sqrt{2}}\n_{E_j}\nu_1-\f{1}{\sqrt{6}}\n_{E_j}(\mu_5\nu_5)+\sqrt{2}\mu_2\lan \n_{E_j} E_1,E_2\ran \nu_2-\sqrt{2}\mu_3 \lan \n_{E_j}E_2,E_3\ran \nu_4,\\
(\n_{E_j}B)_{E_1E_2}=&\f{1}{\sqrt{2}}\n_{E_j}(\mu_2\nu_2)+\sqrt{2}\mu_1\lan \n_{E_j}E_1,E_2\ran \nu_1-\f{\mu_3}{\sqrt{2}}\lan \n_{E_j}E_2,E_3\ran \nu_3-\f{\mu_3}{\sqrt{2}}\lan \n_{E_j}E_1,E_3\ran \nu_4,\\
(\n_{E_j}B)_{E_1E_3}=&\f{1}{\sqrt{2}}\n_{E_j}(\mu_3\nu_3)+\f{\mu_1}{\sqrt{2}}\lan \n_{E_j}E_1,E_3\ran \nu_1+\f{\mu_2}{\sqrt{2}}\lan \n_{E_j}E_2,E_3\ran \nu_2\\
                     &-\f{\mu_3}{\sqrt{2}}\lan \n_{E_j}E_1,E_2\ran \nu_4-\f{3\mu_5}{\sqrt{6}}\lan \n_{E_j}E_1,E_3\ran \nu_5,\\
(\n_{E_j}B)_{E_2E_3}=&\f{1}{\sqrt{2}}\n_{E_j}(\mu_3\nu_4)-\f{\mu_1}{\sqrt{2}}\lan \n_{E_j}E_2,E_3\ran \nu_1+\f{\mu_2}{\sqrt{2}}\lan \n_{E_j}E_1,E_3\ran \nu_2\\
                     &+\f{\mu_3}{\sqrt{2}}\lan \n_{E_j}E_1,E_2\ran \nu_3-\f{3\mu_5}{\sqrt{6}}\lan \n_{E_j}E_2,E_3\ran \nu_5,\\
(\n_{E_j}B)_{E_3E_3}=&\f{2}{\sqrt{6}}\n_{E_j}(\mu_5\nu_5)+\sqrt{2}\mu_3\lan \n_{E_j}E_1,E_3\ran \nu_3+\sqrt{2}\mu_3\lan \n_{E_j}E_2,E_3\ran \nu_4.
\endaligned
\end{equation}
By $h_{12j}^1=h_{13j}^1=h_{23j}^1=0$, we have
\begin{equation}\label{CaseII.1}
\aligned
\lan \n_{E_j}\nu_1,\mu_2 \nu_2\ran=&2\mu_1\lan \n_{E_j}E_1,E_2\ran,\\
\lan \n_{E_j}\nu_1,\mu_3 \nu_3\ran=&\mu_1\lan \n_{E_j}E_1,E_3\ran,\\
\lan \n_{E_j}\nu_1,\mu_3 \nu_4\ran=&-\mu_1\lan \n_{E_j}E_2,E_3\ran.
\endaligned
\end{equation}

$h_{121}^2=h_{112}^2$ says
$$\f{1}{\sqrt{2}}\n_{E_1}(\mu_2)=\f{\mu_1}{\sqrt{2}}\lan \n_{E_2}\nu_1,\nu_2\ran-\f{1}{\sqrt{6}}\lan \n_{E_2}(\mu_5\nu_5),\nu_2\ran-\sqrt{2}\mu_2\lan \n_{E_2}E_1,E_2\ran.$$
In conjunction with (\ref{CaseII.1}), we get
\begin{equation}\label{CaseII.8}
\n_{E_1}(\la_2)=4(\la_1-\la_2)\lan \n_{E_2}E_1,E_2\ran+\f{2}{\sqrt{3}}\lan \n_{E_2}(\mu_2 \nu_2),\mu_5\nu_5\ran.
\end{equation}
Similarly, due to $h_{131}^3=h_{113}^3$, $h_{221}^5=h_{122}^5$, $h_{331}^5=h_{133}^5$, $h_{122}^2=h_{221}^2$, $h_{232}^4=h_{223}^4$, $h_{112}^5=h_{121}^5$, $h_{332}^5=h_{233}^5$, $h_{133}^3=h_{331}^3$,
$h_{233}^4=h_{332}^4$, $h_{113}^5=h_{131}^5$ and $h_{223}^5=h_{232}^5$,  we can derive
\begin{eqnarray}
&&\n_{E_1}(\la_3)=2(\la_1-2\la_3)\lan \n_{E_3}E_1,E_3\ran+\f{2}{\sqrt{3}}\lan \n_{E_3}(\mu_3 \nu_3),\mu_5\nu_5\ran,\label{CaseII.9}\\
&&\n_{E_1}(\la_5)=-2\sqrt{3}\lan \n_{E_2}(\mu_2\nu_2),\mu_5\nu_5\ran,\label{CaseII.3}\\
&&\n_{E_1}(\la_5)=-3\la_5\lan \n_{E_3}E_1,E_3\ran+\sqrt{3}\lan \n_{E_3}(\mu_3\nu_3),\mu_5\nu_5\ran,\label{CaseII.4}\\
&&\n_{E_2}(\la_2)=-4(\la_1-\la_2)\lan \n_{E_1}E_1,E_2\ran+\f{2}{\sqrt{3}}\lan \n_{E_1}(\mu_2 \nu_2),\mu_5\nu_5\ran,\label{CaseII.11}\\
&&\n_{E_2}(\la_3)=2(\la_1-2\la_3)\lan \n_{E_3}E_2,E_3\ran+\f{2}{\sqrt{3}}\lan \n_{E_3}(\mu_3 \nu_4),\mu_5\nu_5\ran,\label{CaseII.12}\\
&&\n_{E_2}(\la_5)=-2\sqrt{3}\lan \n_{E_1}(\mu_2\nu_2),\mu_5\nu_5\ran,\label{CaseII.13}\\
&&\n_{E_2}(\la_5)=-3\la_5\lan \n_{E_3}E_2,E_3\ran+\sqrt{3}\lan \n_{E_3}(\mu_3\nu_4),\mu_5\nu_5\ran,\label{CaseII.14}\\
&&\n_{E_3}(\la_3)=4\la_3 \lan \n_{E_1}E_1,E_3\ran-\f{4}{\sqrt{3}}\lan \n_{E_1}(\mu_3\nu_3),\mu_5\nu_5\ran,\label{CaseII.18}\\
&&\n_{E_3}(\la_3)=4\la_3 \lan \n_{E_2}E_2,E_3\ran-\f{4}{\sqrt{3}}\lan \n_{E_2}(\mu_3\nu_4),\mu_5\nu_5\ran,\label{CaseII.19}\\
&&\n_{E_3}(\la_5)=6\la_5 \lan \n_{E_1}E_1,E_3\ran-\f{6}{\sqrt{3}}\lan \n_{E_1}(\mu_3\nu_3),\mu_5\nu_5\ran,\label{CaseII.20}\\
&&\n_{E_3}(\la_5)=6\la_5 \lan \n_{E_2}E_2,E_3\ran-\f{6}{\sqrt{3}}\lan \n_{E_2}(\mu_3\nu_4),\mu_5\nu_5\ran,\label{CaseII.21}
\end{eqnarray}	
by using (\ref{CaseII.1}) and (\ref{CaseII.2}).
On the other hand, $h_{231}^4=h_{123}^4=h_{132}^4$, $h_{132}^3=h_{123}^3=h_{231}^3$, $h_{123}^2=h_{132}^2=h_{231}^2$  and $h_{123}^5=h_{132}^5=h_{231}^5$ enable us to deduce
\begin{equation}\label{CaseII.7}\aligned
\n_{E_1}(\la_3)=&-2\la_3\lan \n_{E_3}E_1,E_3\ran+2\lan \n_{E_3}(\mu_2\nu_2),\mu_3\nu_4\ran\\
=&-2\la_3\lan \n_{E_2}E_1,E_2\ran+2\lan \n_{E_2}(\mu_3\nu_3),\mu_3\nu_4\ran,
\endaligned
\end{equation}
\begin{equation}\label{CaseII.15}\aligned
\n_{E_2}(\la_3)=&-2\la_3\lan \n_{E_3}E_2,E_3\ran+2\lan \n_{E_3}(\mu_2\nu_2),\mu_3\nu_3\ran\\
=&2\la_3\lan \n_{E_1}E_1,E_2\ran-2\lan \n_{E_1}(\mu_3\nu_3),\mu_3\nu_4\ran,
\endaligned
\end{equation}
\begin{equation}\label{CaseII.22}\aligned
\n_{E_3}(\la_2)=&2\la_2\lan \n_{E_2}E_2,E_3\ran-2\lan \n_{E_2}(\mu_2\nu_2),\mu_3\nu_3\ran\\
=&2\la_2\lan \n_{E_1}E_1,E_3\ran-2\lan \n_{E_1}(\mu_2\nu_2),\mu_3\nu_4\ran,
\endaligned
\end{equation}
and
\begin{equation}\label{CaseII.25}\aligned
\f{1}{\sqrt{3}}\lan \n_{E_3}(\mu_2\nu_2),\mu_5\nu_5\ran=&\f{1}{\sqrt{3}}\lan \n_{E_2}(\mu_3\nu_3),\mu_5\nu_5\ran-\la_5\lan \n_{E_2}E_1,E_3\ran\\
=&\f{1}{\sqrt{3}}\lan \n_{E_1}(\mu_3\nu_4),\mu_5\nu_5\ran-\la_5\lan \n_{E_1}E_2,E_3\ran.
\endaligned
\end{equation}
Afterwards, it follows from $h_{223}^3=h_{232}^3$, $h_{332}^2=h_{233}^2$, $h_{113}^4=h_{131}^4$, $h_{331}^2=h_{133}^2$, $h_{112}^4=h_{121}^4$, $h_{221}^3=h_{122}^3$, $h_{112}^3=h_{121}^3$, $h_{113}^2=h_{311}^2$,
$h_{221}^4=h_{122}^4$, $h_{223}^2=h_{232}^2$, $h_{331}^4=h_{133}^4$ and $h_{332}^3=h_{233}^3$ that
\begin{eqnarray}
\lan \n_{E_2}(\mu_3\nu_3),\mu_3\nu_4\ran+\f{1}{\sqrt{3}}\lan \n_{E_3}(\mu_3\nu_3),\mu_5\nu_5\ran=\la_1\lan \n_{E_3}E_1,E_3\ran+\la_3\lan \n_{E_2}E_1,E_2\ran,\label{CaseII.5}\\
\lan \n_{E_3}(\mu_2\nu_2),\mu_3\nu_4\ran-\f{2}{\sqrt{3}}\lan \n_{E_2}(\mu_2\nu_2),\mu_5\nu_5\ran=\la_2\lan \n_{E_3}E_1,E_3\ran,\label{CaseII.6}\\
\lan \n_{E_1}(\mu_3\nu_3),\mu_3\nu_4\ran-\f{1}{\sqrt{3}}\lan \n_{E_3}(\mu_3\nu_4),\mu_5\nu_5\ran=-\la_1\lan \n_{E_3}E_2,E_3\ran+\la_3\lan \n_{E_1}E_1,E_2\ran,\label{CaseII.16}\\
\lan \n_{E_3}(\mu_2\nu_2),\mu_3\nu_3\ran-\f{2}{\sqrt{3}}\lan \n_{E_1}(\mu_2\nu_2),\mu_5\nu_5\ran=\la_2\lan \n_{E_3}E_2,E_3\ran,\label{CaseII.17}\\
\lan \n_{E_1}(\mu_2\nu_2),\mu_3\nu_4\ran-\f{1}{\sqrt{3}}\lan \n_{E_2}(\mu_3\nu_4),\mu_5\nu_5\ran=-\la_1\lan \n_{E_2}E_2,E_3\ran+\la_3\lan \n_{E_1}E_1,E_3\ran,\label{CaseII.23}\\
\lan \n_{E_2}(\mu_2\nu_2),\mu_3\nu_3\ran-\f{1}{\sqrt{3}}\lan \n_{E_1}(\mu_3\nu_3),\mu_5\nu_5\ran=-\la_1\lan \n_{E_1}E_1,E_3\ran+\la_3\lan \n_{E_2}E_2,E_3\ran,\label{CaseII.24}\\
\lan \n_{E_1}(\mu_2\nu_2),\mu_3\nu_3\ran-\f{1}{\sqrt{3}}\lan \n_{E_2}(\mu_3\nu_3),\mu_5\nu_5\ran=\la_3\lan \n_{E_1}E_2,E_3\ran+(\la_1-2\la_3)\lan \n_{E_2}E_1,E_3\ran,\label{CaseII.26}\\
\lan \n_{E_1}(\mu_2\nu_2),\mu_3\nu_3\ran+\f{1}{\sqrt{3}}\lan \n_{E_3}(\mu_2\nu_2),\mu_5\nu_5\ran=\la_2\lan \n_{E_1}E_2,E_3\ran-2(\la_1-\la_2)\lan \n_{E_3}E_1,E_2\ran,\label{CaseII.27}\\
\lan \n_{E_2}(\mu_2\nu_2),\mu_3\nu_4\ran-\f{1}{\sqrt{3}}\lan \n_{E_1}(\mu_3\nu_4),\mu_5\nu_5\ran=\la_3\lan \n_{E_2}E_1,E_3\ran+(\la_1-2\la_3)\lan \n_{E_1}E_2,E_3\ran,\label{CaseII.28}\\
\lan \n_{E_2}(\mu_2\nu_2),\mu_3\nu_4\ran+\f{1}{\sqrt{3}}\lan \n_{E_3}(\mu_2\nu_2),\mu_5\nu_5\ran=\la_2\lan \n_{E_2}E_1,E_3\ran+2(\la_1-\la_2)\lan \n_{E_3}E_1,E_2\ran,\label{CaseII.29}\\
\lan \n_{E_3}(\mu_3\nu_3),\mu_3\nu_4\ran+\f{2}{\sqrt{3}}\lan \n_{E_1}(\mu_3\nu_4),\mu_5\nu_5\ran=2\la_3\lan \n_{E_1}E_2,E_3\ran+\la_3\lan \n_{E_3}E_1,E_2\ran,\label{CaseII.30}\\
\lan \n_{E_3}(\mu_3\nu_3),\mu_3\nu_4\ran-\f{2}{\sqrt{3}}\lan \n_{E_2}(\mu_3\nu_3),\mu_5\nu_5\ran=\la_3\lan \n_{E_3}E_1,E_2\ran-2\la_3\lan \n_{E_2}E_1,E_3\ran.\label{CaseII.31}
\end{eqnarray}
	
Combining (\ref{CaseII.3}) and (\ref{CaseII.4}) implies
$$\la_5\lan \n_{E_3}E_1,E_3\ran=\f{1}{\sqrt{3}}\lan \n_{E_3}(\mu_3\nu_3),\mu_5\nu_5\ran+\f{2}{\sqrt{3}}\lan \n_{E_2}(\mu_2\nu_2),\mu_5\nu_5\ran.$$
Substituting this equation into (\ref{CaseII.5}) and (\ref{CaseII.6}), we get
$$\lan \n_{E_2}(\mu_3\nu_3),\mu_3\nu_4\ran-\lan \n_{E_3}(\mu_2\nu_2),\mu_3\nu_4\ran=(\la_1-\la_2-\la_5)\lan \n_{E_3}E_1,E_3\ran+\la_3\lan \n_{E_2}E_1,E_2\ran.$$
On the other hand, (\ref{CaseII.7}) means
$$\lan \n_{E_2}(\mu_3\nu_3),\mu_3\nu_4\ran-\lan \n_{E_3}(\mu_2\nu_2),\mu_3\nu_4\ran=-\la_3\lan \n_{E_3}E_1,E_3\ran+\la_3\lan \n_{E_2}E_1,E_2\ran.$$
Thus $(\la_1-\la_2+\la_3-\la_5)\lan \n_{E_3}E_1,E_3\ran=0$, i.e.
\begin{equation}\label{CaseII.33}
\lan \n_{E_3}E_1,E_3\ran=0
\end{equation}
(note that $\la_1>\la_2$ and $\la_3\geq \la_5$). Thereby, in conjunction with (\ref{CaseII.10}), (\ref{CaseII.8}), (\ref{CaseII.9}), (\ref{CaseII.7}), (\ref{CaseII.5}) and (\ref{CaseII.6}), we can derive
\begin{equation}
\lan \n_{E_2}E_1,E_2\ran=0,
\end{equation}
\begin{equation}
\lan \n_{E_3}\nu_3,\nu_5\ran=\lan \n_{E_2}\nu_2,\nu_5\ran=0,
\end{equation}
\begin{equation}
\lan \n_{E_2}\nu_3,\nu_4\ran=\lan \n_{E_3}\nu_2,\nu_4\ran=0
\end{equation}
and
\begin{equation}
\n_{E_1}(\la_2)=\n_{E_1}(\la_3)=\n_{E_1}(\la_5)=0.
\end{equation}
Similarly, based on (\ref{CaseII.10}), (\ref{CaseII.11}), (\ref{CaseII.12}), (\ref{CaseII.13}), (\ref{CaseII.14}), (\ref{CaseII.15}), (\ref{CaseII.16}) and (\ref{CaseII.17}),
we obtain
\begin{equation}
\lan \n_{E_1}E_1,E_2\ran=\lan \n_{E_3}E_2,E_3\ran=0,
\end{equation}
\begin{equation}
\lan \n_{E_3}\nu_4,\nu_5\ran=\lan \n_{E_1}\nu_2,\nu_5\ran=0,
\end{equation}
\begin{equation}
\lan \n_{E_1}\nu_3,\nu_4\ran=\lan \n_{E_3}\nu_2,\nu_3\ran=0
\end{equation}
and
\begin{equation}
\n_{E_2}(\la_2)=\n_{E_2}(\la_3)=\n_{E_2}(\la_5)=0.
\end{equation}
Next, combining (\ref{CaseII.10}), (\ref{CaseII.18}), (\ref{CaseII.19}), (\ref{CaseII.20}), (\ref{CaseII.21}), (\ref{CaseII.22}), (\ref{CaseII.23}) and (\ref{CaseII.24}) yields
\begin{equation}
\lan \n_{E_1}E_1,E_3\ran=\lan \n_{E_2}E_2,E_3\ran,
\end{equation}
\begin{equation}
\lan \n_{E_1}(\mu_3\nu_3),\mu_5\nu_5\ran=\lan \n_{E_1}(\mu_3\nu_4),\mu_5\nu_5\ran=\f{\sqrt{3}(\la_1+\la_2)}{2}\lan \n_{E_1}E_1,E_3\ran,
\end{equation}
\begin{equation}
\lan \n_{E_1}(\mu_2\nu_2),\mu_3\nu_4\ran=\lan \n_{E_2}(\mu_2\nu_2),\mu_3\nu_3\ran=\f{\la_2+2\la_3-\la_1}{2}\lan \n_{E_1}E_1,E_3\ran
\end{equation}
and
\begin{equation}\aligned
\n_{E_3}(\la_2)=&(\la_1+\la_2-2\la_3)\lan \n_{E_1}E_1,E_3\ran,\\
\n_{E_3}(\la_3)=&-2(\la_1+\la_2-2\la_3)\lan \n_{E_1}E_1,E_3\ran,\\
\n_{E_3}(\la_5)=&-3(\la_1+\la_2-2\la_5)\lan \n_{E_1}E_1,E_3\ran.
\endaligned
\end{equation}
Similarly, a straightforward calculation based on (\ref{CaseII.25}), (\ref{CaseII.26}), (\ref{CaseII.27}), (\ref{CaseII.28}), (\ref{CaseII.29}), (\ref{CaseII.30}) and (\ref{CaseII.31})
gives
\begin{equation}
\lan \n_{E_2}E_1,E_3\ran=-\lan \n_{E_1}E_2,E_3\ran,
\end{equation}
\begin{equation}
\lan \n_{E_3}E_1,E_2\ran=\f{\la_1+\la_2-3\la_3+\la_5}{2(\la_1-\la_2)}\lan \n_{E_1}E_2,E_3\ran,
\end{equation}
\begin{equation}
\lan \n_{E_3}\nu_2,\nu_5\ran=0,
\end{equation}
\begin{equation}
\lan \n_{E_2}(\mu_3\nu_3),\mu_5\nu_5\ran=-\lan \n_{E_1}(\mu_3\nu_4),\mu_5\nu_5\ran=-\sqrt{3}\la_5\lan \n_{E_1}E_2,E_3\ran
\end{equation}
\begin{equation}
\lan \n_{E_2}(\mu_2\nu_2),\mu_3\nu_4\ran=-\lan \n_{E_1}(\mu_2\nu_2),\mu_3\nu_3\ran=(\la_1-3\la_3+\la_5)\lan \n_{E_1}E_2,E_3\ran
\end{equation}
and
\begin{equation}
\lan \n_{E_3}(\mu_3\nu_3),\mu_3\nu_4\ran=\la_3 \lan \n_{E_3}E_1,E_2\ran+2(\la_3-\la_5)\lan \n_{E_1}E_2,E_3\ran.
\end{equation}

Let $(\cdots)^\bot$ denote the projection on the orthogonal complement of the subspace spanned by $\nu_1,\cdots,\nu_5$,
then again applying (\ref{dB2}) implies
\begin{equation}\aligned
\big(\n_{E_1}(\mu_2\nu_2)\big)^\bot=&\sqrt{2}(\n_{E_1}B)_{E_1E_2}^\bot=\sqrt{2}(\n_{E_2}B)_{E_1E_1}^\bot\\
=&\mu_1(\n_{E_2}\nu_1)^\bot-\f{1}{\sqrt{3}}\big(\n_{E_2}(\mu_5\nu_5)\big)^\bot,
\endaligned
\end{equation}
\begin{equation}\aligned
\big(\n_{E_2}(\mu_2\nu_2)\big)^\bot=&\sqrt{2}(\n_{E_2}B)_{E_1E_2}^\bot=\sqrt{2}(\n_{E_1}B)_{E_2E_2}^\bot\\
=&-\mu_1(\n_{E_1}\nu_1)^\bot-\f{1}{\sqrt{3}}\big(\n_{E_1}(\mu_5\nu_5)\big)^\bot,
\endaligned
\end{equation}
\begin{equation}\aligned
\big(\n_{E_1}(\mu_3\nu_3)\big)^\bot=&\sqrt{2}(\n_{E_1}B)_{E_1E_3}^\bot=\sqrt{2}(\n_{E_3}B)_{E_1E_1}^\bot\\
=&\mu_1(\n_{E_3}\nu_1)^\bot-\f{1}{\sqrt{3}}\big(\n_{E_3}(\mu_5\nu_5)\big)^\bot,
\endaligned
\end{equation}
\begin{equation}\aligned
\big(\n_{E_3}(\mu_3\nu_3)\big)^\bot=&\sqrt{2}(\n_{E_3}B)_{E_1E_3}^\bot=\sqrt{2}(\n_{E_1}B)_{E_3E_3}^\bot\\
=&\f{2}{\sqrt{3}}\big(\n_{E_1}(\mu_5\nu_5)\big)^\bot,
\endaligned
\end{equation}
\begin{equation}\aligned
\big(\n_{E_2}(\mu_3\nu_4)\big)^\bot=&\sqrt{2}(\n_{E_2}B)_{E_2E_3}^\bot=\sqrt{2}(\n_{E_3}B)_{E_2E_2}^\bot\\
=&-\mu_1(\n_{E_3}\nu_1)^\bot-\f{1}{\sqrt{3}}\big(\n_{E_3}(\mu_5\nu_5)\big)^\bot
\endaligned
\end{equation}
and
\begin{equation}\label{CaseII.34}\aligned
\big(\n_{E_3}(\mu_3\nu_4)\big)^\bot=&\sqrt{2}(\n_{E_3}B)_{E_2E_3}^\bot=\sqrt{2}(\n_{E_2}B)_{E_3E_3}^\bot\\
=&\f{2}{\sqrt{3}}\big(\n_{E_2}(\mu_5\nu_5)\big)^\bot.
\endaligned
\end{equation}

In conjunction with (\ref{CaseII.32}), (\ref{CaseII.1}), (\ref{CaseII.33})-(\ref{CaseII.34}), we have
\begin{equation}\aligned
&\f{1}{2\mu_1}\lan \n_{E_2}\nu_1,\n_{E_1}(\mu_2\nu_2)\ran\\
=&\f{1}{2\mu_1}\sum_{\a=2}^5 \f{1}{\la_\a}\lan \n_{E_2}\nu_1,\mu_\a\nu_\a\ran \lan \n_{E_1}(\mu_2\nu_2),\mu_\a \nu_\a\ran+\f{1}{2\mu_1} \lan (\n_{E_2}\nu_1)^\bot,\big(\n_{E_1}(\mu_2\nu_2)\big)^\bot\ran\\
=&\f{\la_1-3\la_3+\la_5}{2\la_3}\lan \n_{E_1}E_2,E_3\ran^2+\f{\la_1-\la_2-2\la_3}{4\la_3}\lan \n_{E_1}E_1,E_3\ran^2\\
&+\f{1}{2}|(\n_{E_2}\nu_1)^\bot|^2-\f{1}{2\sqrt{3}\mu_1}\lan (\n_{E_2}\nu_1)^\bot, \big(\n_{E_2}(\mu_5\nu_5)\big)^\bot\ran,
\endaligned
\end{equation}
\begin{equation}\aligned
&\f{1}{2\mu_1}\lan \n_{E_1}\nu_1,\n_{E_2}(\mu_2\nu_2)\ran\\
=&\f{1}{2\mu_1}\sum_{\a=2}^5 \f{1}{\la_\a}\lan \n_{E_1}\nu_1,\mu_\a\nu_\a\ran \lan \n_{E_2}(\mu_2\nu_2),\mu_\a \nu_\a\ran+\f{1}{2\mu_1} \lan (\n_{E_1}\nu_1)^\bot,\big(\n_{E_2}(\mu_2\nu_2)\big)^\bot\ran\\
=&-\f{\la_1-3\la_3+\la_5}{2\la_3}\lan \n_{E_1}E_2,E_3\ran^2-\f{\la_1-\la_2-2\la_3}{4\la_3}\lan \n_{E_1}E_1,E_3\ran^2\\
&-\f{1}{2}|(\n_{E_1}\nu_1)^\bot|^2-\f{1}{2\sqrt{3}\mu_1}\lan (\n_{E_1}\nu_1)^\bot, \big(\n_{E_1}(\mu_5\nu_5)\big)^\bot\ran,
\endaligned
\end{equation}
and hence
\begin{equation}\label{CaseII.35}\aligned
&\f{1}{2\mu_1}\lan R_{E_1E_2}\nu_1,\mu_2\nu_2\ran=\f{1}{2\mu_1}\lan -\n_{E_1}\n_{E_2}\nu_1+\n_{E_2}\n_{E_1}\nu_1+\n_{[E_1,E_2]}\nu_1,\mu_2\nu_2\ran\\
=&\f{1}{2\mu_1}\Big(-\n_{E_1}\lan \n_{E_2}\nu_1,\mu_2\nu_2\ran+\n_{E_2}\lan \n_{E_1}\nu_1,\mu_2\nu_2\ran+\lan\n_{[E_1,E_2]}\nu_1,\mu_2\nu_2\ran\Big)\\
&+\f{1}{2\mu_1}\lan \n_{E_2}\nu_1,\n_{E_1}(\mu_2\nu_2)\ran-\f{1}{2\mu_1}\lan \n_{E_1}\nu_1,\n_{E_2}(\mu_2\nu_2)\ran\\
=&-\n_{E_1}\lan \n_{E_2}E_1,E_2\ran+\n_{E_2}\lan \n_{E_1}E_1,E_2\ran+\lan\n_{[E_1,E_2]}E_1,E_2\ran\\
&+\f{1}{2\mu_1}\lan \n_{E_2}\nu_1,\n_{E_1}(\mu_2\nu_2)\ran-\f{1}{2\mu_1}\lan \n_{E_1}\nu_1,\n_{E_2}(\mu_2\nu_2)\ran\\
=&R_{1212}-\lan \n_{E_2}E_1,\n_{E_1}E_2\ran+\lan \n_{E_1}E_1,\n_{E_2}E_2\ran+\f{1}{2\mu_1}\lan \n_{E_2}\nu_1,\n_{E_1}(\mu_2\nu_2)\ran-\f{1}{2\mu_1}\lan \n_{E_1}\nu_1,\n_{E_2}(\mu_2\nu_2)\ran\\
=&R_{1212}+\f{\la_1-2\la_3+\la_5}{\la_3}\lan \n_{E_1}E_2,E_3\ran^2+\f{\la_1-\la_2}{2\la_3}\lan \n_{E_1}E_1,E_3\ran^2
+\f{1}{2}(|(\n_{E_2}\nu_1)^\bot|^2+|(\n_{E_1}\nu_1)^\bot|^2)\\
&-\f{1}{2\sqrt{3}\mu_1}\lan (\n_{E_2}\nu_1)^\bot, \big(\n_{E_2}(\mu_5\nu_5)\big)^\bot\ran+\f{1}{2\sqrt{3}\mu_1}\lan (\n_{E_1}\nu_1)^\bot, \big(\n_{E_1}(\mu_5\nu_5)\big)^\bot\ran.
\endaligned
\end{equation}
Similarly, direct calculations show
\begin{equation}\label{CaseII.36}\aligned
&\f{1}{\mu_1}\lan R_{E_1E_3}\nu_1,\mu_3\nu_3\ran\\
=&R_{1313}-\lan \n_{E_3}E_1,\n_{E_1}E_3\ran+\lan \n_{E_1}E_1,\n_{E_3}E_3\ran+\f{1}{\mu_1}\lan \n_{E_3}\nu_1,\n_{E_1}(\mu_3\nu_3)\ran-\f{1}{\mu_1}\lan \n_{E_1}\nu_1,\n_{E_3}(\mu_3\nu_3)\ran\\
=&R_{1313}+\f{4(\la_1-\la_2)}{\la_2}\lan \n_{E_3}E_1,E_2\ran^2+\f{2(\la_3-\la_5)}{\la_3}\lan \n_{E_1}E_2,E_3\ran^2+\f{\la_1+\la_2-2\la_3}{\la_3}\lan \n_{E_1}E_1,E_3\ran^2\\
&+|(\n_{E_3}\nu_1)^\bot|^2-\f{1}{\sqrt{3}\mu_1}\lan (\n_{E_3}\nu_1)^\bot, \big(\n_{E_3}(\mu_5\nu_5)\big)^\bot\ran-\f{2}{\sqrt{3}\mu_1}\lan (\n_{E_1}\nu_1)^\bot, \big(\n_{E_1}(\mu_5\nu_5)\big)^\bot\ran
\endaligned
\end{equation}
and
\begin{equation}\label{CaseII.37}\aligned
&-\f{1}{\mu_1}\lan R_{E_2E_3}\nu_1,\mu_3\nu_4\ran\\
=&R_{2323}-\lan \n_{E_3}E_2,\n_{E_2}E_3\ran+\lan \n_{E_2}E_2,\n_{E_3}E_3\ran-\f{1}{\mu_1}\lan \n_{E_3}\nu_1,\n_{E_2}(\mu_3\nu_4)\ran+\f{1}{\mu_1}\lan \n_{E_2}\nu_1,\n_{E_3}(\mu_3\nu_4)\ran\\
=&R_{2323}+\f{4(\la_1-\la_2)}{\la_2}\lan \n_{E_3}E_1,E_2\ran^2+\f{2(\la_3-\la_5)}{\la_3}\lan \n_{E_1}E_2,E_3\ran^2+\f{\la_1+\la_2-2\la_3}{\la_3}\lan \n_{E_1}E_1,E_3\ran^2\\
&+|(\n_{E_3}\nu_1)^\bot|^2+\f{1}{\sqrt{3}\mu_1}\lan (\n_{E_3}\nu_1)^\bot, \big(\n_{E_3}(\mu_5\nu_5)\big)^\bot\ran+\f{2}{\sqrt{3}\mu_1}\lan (\n_{E_2}\nu_1)^\bot, \big(\n_{E_2}(\mu_5\nu_5)\big)^\bot\ran.
\endaligned
\end{equation}
Combining (\ref{CaseII.35})-(\ref{CaseII.37}), we get
\begin{equation}\label{CaseII.38}
\aligned
&\left(\f{1}{\mu_1}\lan R_{E_1E_2}\nu_1,\mu_2\nu_2\ran-2R_{1212}\right)+\left(\f{1}{2\mu_1}\lan R_{E_1E_3}\nu_1,\mu_3\nu_3\ran-\f{1}{2}R_{1313}\right)\\
&+\left(-\f{1}{2\mu_1}\lan R_{E_2E_3}\nu_1,\mu_3\nu_4\ran-\f{1}{2}R_{2323}\right)\\
=&\f{2(\la_1-\la_3)}{\la_3}\lan \n_{E_1}E_1,E_3\ran^2+\f{2(\la_1-\la_3)}{\la_3}\lan \n_{E_1}E_2,E_3\ran^2+\f{4(\la_1-\la_2)}{\la_2}\lan \n_{E_3}E_1,E_2\ran^2\\
&+\sum_{j=1}^3|(\n_{E_j}\nu_1)^\bot|^2.
\endaligned
\end{equation}

By Gauss equations and Ricci equations, we can derive
\begin{equation}\label{CaseII.39}\aligned
R_{1212}&=\lan \ol{R}_{E_1E_2}E_1,E_2\ran+\lan B_{E_1E_1},B_{E_2E_2}\ran-\lan B_{E_1E_2},B_{E_1E_2}\ran\\
&=1-\f{\la_1}{2}+\f{\la_5}{6}-\f{\la_2}{2},\\
R_{1313}&=\lan \ol{R}_{E_1E_3}E_1,E_3\ran+\lan B_{E_1E_1},B_{E_3E_3}\ran-\lan B_{E_1E_3},B_{E_1E_3}\ran\\
&=1-\f{\la_5}{3}-\f{\la_3}{2},\\
R_{2323}&=\lan \ol{R}_{E_2E_3}E_2,E_3\ran+\lan B_{E_2E_2},B_{E_3E_3}\ran-\lan B_{E_2E_3},B_{E_2E_3}\ran\\
&=1-\f{\la_5}{3}-\f{\la_3}{2},
\endaligned
\end{equation}
\begin{equation}\aligned
\lan R_{E_1E_2}\nu_1,\nu_2\ran=&\lan \ol{R}_{E_1E_2}\nu_1,\nu_2\ran+\lan A^1(E_1),A^2(E_2)\ran-\lan A^1(E_2),A^2(E_1)\ran=\mu_1\mu_2,\\
\lan R_{E_1E_3}\nu_1,\nu_3\ran=&\lan \ol{R}_{E_1E_3}\nu_1,\nu_3\ran+\lan A^1(E_1),A^2(E_3)\ran-\lan A^1(E_3),A^3(E_1)\ran=\f{\mu_1\mu_3}{2},\\
\lan R_{E_2E_3}\nu_1,\nu_4\ran=&\lan \ol{R}_{E_2E_3}\nu_1,\nu_4\ran+\lan A^1(E_2),A^4(E_3)\ran-\lan A^1(E_3),A^4(E_2)\ran=-\f{\mu_1\mu_3}{2}
\endaligned
\end{equation}
and hence
\begin{equation}
\text{The LHS of (\ref{CaseII.38})}=\la_1+2\la_2+\la_3-3=0.
\end{equation}
Thereby
\begin{equation}
\lan \n_{E_1}E_1,E_3\ran=\lan \n_{E_1}E_2,E_3\ran=\lan \n_{E_3}E_1,E_2\ran=0
\end{equation}
and thus $R_{1313}=0$. On the other hand, $\la_1+2\la_2+\la_3=3$ and $\la_1>\la_2>\la_3\geq \la_5$ force
$\la_5\leq \la_3<\f{3}{4}$. Substituting it into (\ref{CaseII.39}) implies $R_{1313}>\f{3}{8}$, causing a contradiction. This means $D_3$
is an empty set, i.e. $\la_5\equiv 0$ on $D_2$. Afterwards, for any $x\in D_2$, one can proceed similarly as above to compute
$\big(\f{1}{\mu_1}\lan R_{E_1E_2}\nu_1,\mu_2\nu_2\ran-2R_{1212}\big)+\big(\f{1}{2\mu_1}\lan R_{E_1E_3}\nu_1,\mu_3\nu_3\ran-\f{1}{2}R_{1313}\big)+\big(-\f{1}{2\mu_1}\lan R_{E_2E_3}\nu_1,\mu_3\nu_4\ran-\f{1}{2}R_{2323}\big)$,
 obtaining a contradiction, so $D_2$ has to be an empty set.

Therefore, $\la_3\equiv 0$ on $D_1$. For any $x\in D_1$, there exists a local orthonormal tangent frame field $E_1,E_2,E_3$ and unit normal fields $\nu_1,\nu_2$
that are orthogonal to each other, such that
\begin{equation}\label{CaseII.49}
\aligned
&B_{E_1E_1}=\f{\mu_1}{\sqrt{2}}\nu_1,\ B_{E_1E_2}=\f{\mu_2}{\sqrt{2}}\nu_2,\ B_{E_2E_2}=-\f{\mu_1}{\sqrt{2}}\nu_1,\\
&B_{E_1E_3}=B_{E_2E_3}=B_{E_3E_3}=0.
\endaligned
\end{equation}
It immediately follows from $h_{11j}^1=0$ that $\n_{E_j}\mu_1=0$, so both $\la_1$ and $\la_2$ are constants on $D_1$.
By $h_{12j}^1=h_{13j}^1=h_{23j}^1=0$, we have
\begin{equation}\label{CaseII.40}
\lan \n_{E_j}\nu_1,\mu_2\nu_2\ran=2\mu_1\lan \n_{E_j}E_1,E_2\ran
\end{equation}
and
\begin{equation}
\lan \n_{E_j}E_1,E_3\ran=\lan\n_{E_j}E_2,E_3\ran=0.
\end{equation}
Due to $h_{121}^2=h_{112}^2$, $h_{122}^2=h_{221}^2$ and $h_{131}^2=h_{113}^2$, in conjunction with (\ref{CaseII.40}), we have
\begin{equation}\label{CaseII.50}
\lan \n_{E_2}E_1,E_2\ran=\lan \n_{E_1}E_1,E_2\ran= \lan \n_{E_3}E_1,E_2\ran=0
\end{equation}
and hence $R_{1313}=0$. On the other hand, the Gauss equations shows $R_{1313}=1$, causing a contradiction.

Therefore $D_1=\emptyset$, i.e. $\la_1=\la_2$ holds everywhere on $\Om_2$. Let
\begin{equation}\aligned
D_4:=\{x\in \Om_2:\la_3>0\},\\
D_5:=\{x\in D_4:\la_5>0\},\\
D_6:=\{x\in D_5:\la_3>\la_5\},
\endaligned
\end{equation}
then $D_6\subset D_5\subset D_4\subset \Om_2$ and all of them are open subsets of $M$.
For any $x\in D_5$, there exists an orthonormal frame field $E_1,E_2,E_3$ and an orthonormal normal
frame field $\nu_1,\cdots,\nu_m$ on a neighborhood of $x$, such that (\ref{B2}) holds and
\begin{equation}
h_{ijk}^1=h_{ijk}^2=0\qquad \forall 1\leq i,j,k\leq 3.
\end{equation}
Note that here, at each considered point, $E_3$ is the unique unit normal vector (up to a sign) satisfying
\begin{equation}
A^\nu(E_3)=0\qquad \forall \nu \text{ s.t. }\lan A^\nu,A^\nu\ran=\la_1|\nu|^2,
\end{equation}
and $E_1,E_2$ can be taken to be an arbitrary orthornormal basis of the orthogonal complement of $E_3$.

Similarly as above, $h_{11j}^1=h_{22j}^1=0$ enable us to derive
\begin{equation}
\lan \n_{E_j}\nu_1,\nu_5\ran=0
\end{equation}
and $\mu_1\equiv \text{const}$, which implies $\la_1,\la_3$ are both constant on $D_5$.
By $h_{12j}^1=h_{13j}^1=h_{23j}^1=h_{13j}^2=h_{23j}^2=h_{33j}^2=0$,
we get
\begin{equation}\label{CaseII.41}
\aligned
\lan \n_{E_j}\nu_1,\nu_2\ran=&2\lan \n_{E_j}E_1,E_2\ran,\\
\lan \n_{E_j}\nu_1,\mu_3\nu_3\ran=&\mu_1\lan \n_{E_j}E_1,E_3\ran,\\
\lan \n_{E_j}\nu_1,\mu_3\nu_4\ran=&-\mu_1\lan \n_{E_j}E_2,E_3\ran,\\
\lan \n_{E_j}\nu_2,\mu_3\nu_3\ran=&\mu_1\lan \n_{E_j}E_2,E_3\ran,\\
\lan \n_{E_j}\nu_2,\mu_3\nu_4\ran=&\mu_1\lan \n_{E_j}E_1,E_3\ran,\\
\lan \n_{E_j}\nu_2,\nu_5\ran=&0.
\endaligned
\end{equation}
Combining with $h_{231}^4=h_{123}^4$ and (\ref{CaseII.41}) yields
$$\la_3\lan \n_{E_3}E_1,E_3\ran=\lan \n_{E_3}(\mu_1\nu_2),\mu_3\nu_4\ran=\la_1\lan \n_{E_3}E_1,E_3\ran,$$
i.e.
\begin{equation}
\lan \n_{E_3}E_1,E_3\ran=0.
\end{equation}
Similarly, from $h_{132}^3=h_{123}^3$, we can deduce
\begin{equation}
\lan \n_{E_3}E_2,E_3\ran=0.
\end{equation}
So
\begin{equation}\label{CaseII.44}
\n_{E_3}E_3=0,
\end{equation}
i.e. the integral curve $\xi$ of $E_3$ is a geodesic of $M$. Thereby, without loss of generality
we can assume $E_1,E_2$ are both parallel along $\xi$, i.e.
\begin{equation}\label{CaseII.43}
\n_{E_3}E_1=\n_{E_3}E_2=0.
\end{equation}

$h_{133}^3=h_{331}^3$ and $h_{233}^4=h_{332}^4$ respectively imply
\begin{equation}\aligned
\la_3\lan \n_{E_1}E_1,E_3\ran=&\f{1}{\sqrt{3}}\lan \n_{E_1}(\mu_3\nu_3),\mu_5\nu_5\ran,\\
\la_3\lan \n_{E_2}E_2,E_3\ran=&\f{1}{\sqrt{3}}\lan \n_{E_2}(\mu_3\nu_4),\mu_5\nu_5\ran.
\endaligned
\end{equation}
In conjunction with $h_{112}^4=h_{121}^4$, we have
$$\aligned
&-\la_1\lan \n_{E_2}E_2,E_3\ran+\la_3\lan \n_{E_1}E_1,E_3\ran\\
=&\lan \n_{E_1}(\mu_1\nu_2),\mu_3\nu_4\ran-\f{1}{\sqrt{3}}\lan \n_{E_2}(\mu_3\nu_4),\mu_5\nu_5\ran\\
=&\la_1\lan \n_{E_1}E_1,E_3\ran-\la_3\lan \n_{E_2}E_2,E_3\ran,
\endaligned$$
which is equivalent to
\begin{equation}
\lan \n_{E_2}E_2,E_3\ran=-\lan \n_{E_1}E_1,E_3\ran
\end{equation}
(since $\la_1>\la_3$). By $h_{113}^5=h_{131}^5$ and $h_{223}^5=h_{232}^5$, we get
$$\aligned
\n_{E_3}(\la_5)=&6\la_5\lan \n_{E_1}E_1,E_3\ran-\f{6}{\sqrt{3}}\lan \n_{E_1}(\mu_3\nu_3),\mu_5\nu_5\ran\\
=&6(\la_5-\la_3)\lan \n_{E_1}E_1,E_3\ran
\endaligned$$
and
$$\aligned
\n_{E_3}(\la_5)=&6\la_5\lan \n_{E_2}E_2,E_3\ran-\f{6}{\sqrt{3}}\lan \n_{E_2}(\mu_3\nu_4),\mu_5\nu_5\ran\\
=&6(\la_5-\la_3)\lan \n_{E_2}E_2,E_3\ran.
\endaligned$$
Adding both sides of the above 2 equations gives
\begin{equation}
\n_{E_3}\la_5=0
\end{equation}
and moreover
\begin{equation}\label{CaseII.45}
\lan \n_{E_1}E_1,E_3\ran=\lan \n_{E_2}E_2,E_3\ran=0
\end{equation}
whenever $x\in D_6$.

By $h_{123}^5=h_{132}^5=h_{231}^5$ and $h_{112}^3=h_{121}^3$ we have
\begin{equation}\label{CaseII.42}\aligned
\la_5\lan \n_{E_2}E_1,E_3\ran=&\f{1}{\sqrt{3}}\lan \n_{E_2}(\mu_3\nu_3),\mu_5\nu_5\ran,\\
\la_5\lan \n_{E_1}E_2,E_3\ran=&\f{1}{\sqrt{3}}\lan \n_{E_1}(\mu_3\nu_4),\mu_5\nu_5\ran
\endaligned
\end{equation}
and
$$\aligned
&\la_3\lan \n_{E_1}E_2,E_3\ran+(\la_1-2\la_3)\lan \n_{E_2}E_1,E_3\ran\\
=&\lan \n_{E_1}(\mu_1\nu_2),\mu_3\nu_3\ran-\f{1}{\sqrt{3}}\lan \n_{E_2}(\mu_3\nu_3),\mu_5\nu_5\ran\\
=&\la_1\lan \n_{E_1}E_2,E_3\ran-\la_5\lan \n_{E_2}E_1,E_3\ran
\endaligned$$
i.e.
\begin{equation}
(\la_1-\la_3)\lan \n_{E_1}E_2,E_3\ran=(\la_1-2\la_3+\la_5)\lan \n_{E_2}E_1,E_3\ran.
\end{equation}
Similarly, combining $h_{221}^4=h_{122}^4$ and (\ref{CaseII.42}), we get
\begin{equation}
(\la_1-\la_3)\lan \n_{E_2}E_1,E_3\ran=(\la_1-2\la_3+\la_5)\lan \n_{E_1}E_2,E_3\ran.
\end{equation}
Therefore
\begin{equation}\label{CaseII.46}
\begin{cases}
\lan\n_{E_2}E_1,E_3\ran+\lan \n_{E_1}E_2,E_3\ran=0 & \text{when }x\in D_6,\\
\lan \n_{E_2}E_1,E_3\ran=\lan \n_{E_1}E_2,E_3\ran & \text{when }x\notin D_6.
\end{cases}
\end{equation}
Especially, $\lan \n_{E_2}E_1,E_3\ran=\lan \n_{E_1}E_2,E_3\ran=0$ whenever $\la_3>\la_5$ (i.e. $x\in D_6$) and $2\la_1-3\la_3+\la_5\neq 0$.
Afterwards, due to $h_{331}^4=h_{131}^4$, we can derive
\begin{equation}\label{CaseII.47}
\lan \n_{E_3}(\mu_3\nu_3),\mu_3\nu_4\ran=2(\la_3-\la_5)\lan \n_{E_1}E_2,E_3\ran.
\end{equation}
with the aid of (\ref{CaseII.43}) and (\ref{CaseII.42}).

When $x\in D_6$, based on (\ref{CaseII.44}), (\ref{CaseII.43}), (\ref{CaseII.45}), (\ref{CaseII.46}) and (\ref{CaseII.47}), a similar calculation as above shows
\begin{equation}
\aligned
&\f{1}{2}\lan R_{E_1E_2}\nu_1,\nu_2\ran\\
=&R_{1212}+\f{\la_3-\la_1}{\la_3}\lan \n_{E_1}E_2,E_3\ran^2
+\f{1}{2}(|(\n_{E_2}\nu_1)^\bot|^2+|(\n_{E_1}\nu_1)^\bot|^2)\\
&-\f{1}{2\sqrt{3}\mu_1}\lan (\n_{E_2}\nu_1)^\bot, \big(\n_{E_2}(\mu_5\nu_5)\big)^\bot\ran+\f{1}{2\sqrt{3}\mu_1}\lan (\n_{E_1}\nu_1)^\bot, \big(\n_{E_1}(\mu_5\nu_5)\big)^\bot\ran,
\endaligned
\end{equation}
\begin{equation}
\aligned
&\f{1}{\mu_1}\lan R_{E_1E_3}\nu_1,\mu_3\nu_3\ran\\
=&R_{1313}+\f{2(\la_3-\la_5)}{\la_3}\lan \n_{E_1}E_2,E_3\ran^2
+(|(\n_{E_3}\nu_1)^\bot|^2\\
&-\f{1}{\sqrt{3}\mu_1}\lan (\n_{E_3}\nu_1)^\bot, \big(\n_{E_3}(\mu_5\nu_5)\big)^\bot\ran-\f{2}{\sqrt{3}\mu_1}\lan (\n_{E_1}\nu_1)^\bot, \big(\n_{E_1}(\mu_5\nu_5)\big)^\bot\ran,
\endaligned
\end{equation}
\begin{equation}
\aligned
&-\f{1}{\mu_1}\lan R_{E_2E_3}\nu_1,\mu_3\nu_4\ran\\
=&R_{2323}+\f{2(\la_3-\la_5)}{\la_3}\lan \n_{E_1}E_2,E_3\ran^2
+(|(\n_{E_3}\nu_1)^\bot|^2\\
&+\f{1}{\sqrt{3}\mu_1}\lan (\n_{E_3}\nu_1)^\bot, \big(\n_{E_3}(\mu_5\nu_5)\big)^\bot\ran+\f{2}{\sqrt{3}\mu_1}\lan (\n_{E_2}\nu_1)^\bot, \big(\n_{E_2}(\mu_5\nu_5)\big)^\bot\ran
\endaligned
\end{equation}
and hence
\begin{equation}\label{CaseII.48}
\aligned
&\big(\lan R_{E_1E_2}\nu_1,\nu_2\ran-2R_{1212}\big)+\left(\f{1}{2\mu_1}\lan R_{E_1E_3}\nu_1,\mu_3\nu_3\ran-\f{1}{2}R_{1313}\right)\\
&+\left(-\f{1}{2\mu_1}\lan R_{E_2E_3}\nu_1,\mu_3\nu_4\ran-\f{1}{2}R_{2323}\right)\\
=&\f{2(2\la_3-\la_1-\la_5)}{\la_3}\lan \n_{E_1}E_2,E_3\ran^2+\sum_{j=1}^3|(\n_{E_j}\nu_1)^\bot|^2\\
=&\f{2(\la_1-\la_3)}{\la_3}\lan \n_{E_1}E_2,E_3\ran^2+\sum_{j=1}^3|(\n_{E_j}\nu_1)^\bot|^2.
\endaligned
\end{equation}
Again using the Gauss equations and the Ricci equations, we know the LHS of (\ref{CaseII.48}) is $0$, hence
$\lan \n_{E_1}E_2,E_3\ran=0$ and then a direct calculation gives $R_{1313}=0$, which contradicts to
$$\aligned
R_{1313}=&\lan \ol{R}_{E_1E_3}E_1,E_3\ran+\lan B_{E_1E_1},B_{E_3E_3}\ran-\lan B_{E_1E_3},B_{E_1E_3}\ran
=&1-\f{\la_5}{3}-\f{\la_3}{2}>\f{3}{8}.
\endaligned$$
Thus $D_6=\emptyset$, i.e. $\la_3=\la_5$ everywhere on $D_5$. By (\ref{CaseII.44}), (\ref{CaseII.43}) and (\ref{CaseII.46}), we have
\begin{equation}
\aligned
R_{3131}=&\lan -\n_{E_3}\n_{E_1}E_3+\n_{E_1}\n_{E_3}E_3+\n_{[E_3,E_1]}E_3,E_1\ran\\
=&-\n_{E_3}\lan \n_{E_1}E_3,E_1\ran+\lan \n_{E_1}E_3,\n_{E_3}E_1\ran\\
&+\lan [E_3,E_1],E_1\ran\lan \n_{E_1}E_3,E_1\ran+\lan [E_3,E_1],E_2\ran\lan \n_{E_2}E_3,E_1\ran\\
=&-\n_{E_3}\lan \n_{E_1}E_3,E_1\ran-\lan \n_{E_1}E_3,E_1\ran^2-\lan \n_{E_1}E_2,E_3\ran^2
\endaligned
\end{equation}
and similarly
\begin{equation}
R_{3232}=-\n_{E_3}\lan \n_{E_2}E_3,E_2\ran-\lan \n_{E_2}E_3,E_2\ran^2-\lan \n_{E_1}E_2,E_3\ran^2
\end{equation}
Adding both hand sides of the above 2 equations, then applying yields $R_{3131}+R_{3232}\leq 0$,
which cause a contradiction to $R_{1313}=R_{2323}=1-\f{\la_5}{3}-\f{\la_3}{2}>\f{3}{8}$. Thus
$D_5$ has to be an empty set, i.e. $\la_5\equiv 0$ on $D_4$. For any $x\in D_4$, we can proceed as above
to compute $\big(\lan R_{E_1E_2}\nu_1,\nu_2\ran-2R_{1212}\big)+\big(\f{1}{2\mu_1}\lan R_{E_1E_3}\nu_1,\mu_3\nu_3\ran-\f{1}{2}R_{1313}\big)$,
obtaining a contradiction, so $D_4=\emptyset$, i.e. $\la_3\equiv 0$ on $\Om_2$. In this case, a calculation as in
(\ref{CaseII.49})-(\ref{CaseII.50}) gives $R_{1313}=0$, causing a contradiction to the corollary of the Gauss equations.

In a summary, $\Om_2$ has to be empty, which is equivalent to the following result.
\begin{lem}\label{CaseII}
Let $M$ be a 3-dimensional compact SCDK-ideal manifold in $S^{3+m}$, then $\la_2=\la_3$ holds everywhere on $M$.
\end{lem}

\subsection{Case III: $\la_2=\la_3=\la_4$}

As shown in Lemma \ref{CaseI} and Lemma \ref{CaseII}, $\la_2=\la_3=\la_4$ holds at each point of $M$, i.e. $\Om_3=M$.
Then there exists $S\subset M$ of measure $0$, such that for each $x\notin S$, there exists an orthonormal basis
$e_1,e_2,e_3$ of $T_x M$ and an orthonormal basis $\nu_1,\cdots,\nu_m$ of $N_x M$, such that
\begin{equation}\label{CaseIII.1}\aligned
A^1=\mu_1\text{diag}(\eta_1,\eta_2,\eta_3),
\qquad\qquad&
A^2=\f{\mu_2}{\sqrt{2}}(E_{12}+E_{21})\\
A^3=\f{\mu_2}{\sqrt{2}}(E_{13}+E_{31}),\qquad\qquad&
A^4=\f{\mu_2}{\sqrt{2}}(E_{23}+E_{32}),\\
A^5=\f{\mu_5}{\sqrt{3}}\text{diag}(\eta_2-\eta_3,\eta_3-\eta_1,\eta_1-\eta_2),\quad&
A^\a=0\ (\forall \a\geq 6).
\endaligned
\end{equation}
Then
\begin{equation}\aligned
&[A^1,A^2]=\f{\mu_1\mu_2}{\sqrt{2}}(\eta_1-\eta_2)(E_{12}-E_{21}),\quad [A^5,A^2]=-\f{3\mu_2\mu_5}{\sqrt{6}}\eta_3(E_{12}-E_{21}),\\
&[A^1,A^3]=\f{\mu_1\mu_2}{\sqrt{2}}(\eta_1-\eta_3)(E_{13}-E_{31}),\quad [A^5,A^3]=-\f{3\mu_2\mu_5}{\sqrt{6}}\eta_2(E_{31}-E_{13}),\\
&[A^1,A^4]=\f{\mu_1\mu_2}{\sqrt{2}}(\eta_2-\eta_3)(E_{23}-E_{32}),\quad [A^5,A^3]=-\f{3\mu_2\mu_5}{\sqrt{6}}\eta_1(E_{23}-E_{32}),\\
&[A^2,A^3]=\f{\la_2}{2}(E_{23}-E_{32}),\ [A^3,A^4]=\f{\la_2}{2}(E_{12}-E_{21}),\ [A^4,A^2]=\f{\la_2}{2}(E_{31}-E_{13})
\endaligned
\end{equation}
and $[A^1,A^5]=0$. Hence
\begin{equation}
\sum_{\a,\be}\|[A^\a,A^\be]\|^2
=6\la_1\la_2+6\la_2\la_5+3\la_2^2.
\end{equation}
In conjunction with
\begin{equation}
3=\sum_{\be=1}^3 \la_\be+\la_2=\la_1+3\la_2
\end{equation}
and $\la_1\geq \la_2\geq \la_5\geq 0$,
we have
\begin{equation}
\aligned
&n|B|^2-\sum_{\a,\be}\lan A^\a,A^\be\ran-\sum_{\a,\be} \|[A^\a,A^\be\|^2\\
=&3\sum_{\a=1}^m \la_\a-\sum_{\a=1}^m \la_\a^2-\sum_{\a,\be} \|[A^\a,A^\be\|^2\\
=&3\la_2^2+3\la_5-\la_5^2-6\la_2\la_5\\
\geq &\min\{3\la_2^2,3\la_2-4\la_2^2\}\geq 0,
\endaligned
\end{equation}
where the equality holds if and only if $\la_2=0$ or $\la_1=\la_2=\la_5=\f{3}{4}$.
Applying (\ref{LaB2}), we have
\begin{equation}\aligned
0=&\f{1}{2}\int_M \De |B|^2*1=\int_M \left(|\n B|^2+n|B|^2-\sum_{\a,\be}\lan A^\a,A^\be\ran-\sum_{\a,\be} \|[A^\a,A^\be\|^2\right)*1\\
=&\int_M \left(|\n B|^2+3\la_2^2+3\la_5-\la_5^2-6\la_2\la_5\right)*1\geq 0.
\endaligned
\end{equation}
This forces $\n B\equiv 0$ and $(\la_1,\la_2,\la_5)\equiv (3,0,0)$ or $(\f{3}{4},\f{3}{4},\f{3}{4})$ on $M$.
If $\la_1\equiv 3$ and $\la_\a\equiv 0$ for all $\a\geq 2$, it is shown in \cite{C-D-K} that
$M$ is a generalized Clifford torus. Now we consider the latter case.

Let $M$ be a 3-dimensional compact SCDK-ideal manifold in $S^{3+m}$, with
$\la_1\equiv \cdots\equiv \la_5\equiv \f{3}{4}$ and $\la_\a\equiv 0$ for each $\a\geq 6$.
Then (\ref{CaseIII.1}) is equivalent to saying that
\begin{equation}
\aligned
&B_{e_1e_2}=\f{\sqrt{6}}{4}\nu_2,\quad B_{e_1e_3}=\f{\sqrt{6}}{4}\nu_3,\quad B_{e_2e_3}=\f{\sqrt{6}}{4}\nu_4,\\
&B_{e_1e_1}=\f{\sqrt{3}}{2}\eta_1\nu_1+\f{1}{2}(\eta_2-\eta_3)\nu_5,\ B_{e_2e_2}=\f{\sqrt{3}}{2}\eta_2\nu_1+\f{1}{2}(\eta_3-\eta_1)\nu_5,\\
&B_{e_3e_3}=\f{\sqrt{3}}{2}\eta_3\nu_1+\f{1}{2}(\eta_1-\eta_2)\nu_5.
\endaligned
\end{equation}
Hence
\begin{equation}\label{CaseIII.2}
\aligned
&|B_{e_1e_1}|=|B_{e_2e_2}|=|B_{e_3e_3}|=\f{\sqrt{2}}{2},\\
&|B_{e_1e_2}|=|B_{e_1e_3}|=|B_{e_2e_3}|=\f{\sqrt{6}}{4},\\
&\lan B_{e_1e_1},B_{e_2e_2}\ran=\lan B_{e_2e_2},B_{e_3e_3}\ran=\lan B_{e_3e_3},B_{e_1e_1}\ran=-\f{1}{4},\\
&\lan B_{e_ie_j},B_{e_ke_l}\ran=0\text{ whenever }i\neq j\text{ and }\{i,j\}\neq \{k,l\}.
\endaligned
\end{equation}
It follows from the Gauss equations that
\begin{equation}\aligned
R_{ijkl}=&\lan \ol{R}_{e_ie_j}e_k,e_l\ran+\lan B_{e_ie_k},B_{e_je_l}\ran-\lan B_{e_ie_l},B_{e_je_k}\ran\\
=&\f{3}{8}(\de_{ik}\de_{jl}-\de_{il}\de_{jk}),
\endaligned
\end{equation}
i.e. $M$ is a $3$-dimensional space form of constant curvature $\f{3}{8}$.
For an arbitrary unit vector $v\in T_x M$, there exists $w:=\cos\th\ e_1+\sin\th\ e_2$
and $\g\in [-\f{\pi}{2},\f{\pi}{2}]$, such that $v=\cos\g\ w+\sin\g\ e_3$, then
\begin{equation}
\aligned
B_{ww}=&(\cos^2\th\ B_{e_1e_1}+\sin^2\th\ B_{e_2e_2})+2\cos\th\sin\th\ B_{e_1e_2},\\
B_{we_3}=&\cos\th\ B_{e_1e_3}+\sin\th\ B_{e_2e_3},\\
B_{vv}=&(\cos^2\g\ B_{ww}+\sin^2\g\ B_{e_3e_3})+2\cos\g\sin\g\ B_{we_3}
\endaligned
\end{equation}
and a calculation based on \ref{CaseIII.2} shows
\begin{equation}
\aligned
&\lan B_{ww},B_{ww}\ran=\f{1}{2},\qquad \lan B_{ww},B_{e_3e_3}\ran=-\f{1}{4},\\
&\lan B_{we_3},B_{we_3}\ran=\f{3}{8},\qquad \lan B_{ww},B_{we_3}\ran=\lan B_{e_3e_3},B_{we_3}\ran=0
\endaligned
\end{equation}
and hence
\begin{equation}
|B_{vv}|=\f{\sqrt{2}}{2},
\end{equation}
i.e. $M$ is an isotropic submanifold of $S^{3+m}$ in the sense of Itoh-Ogiue \cite{I-O}. Applying the main theorems of this paper, we see $M$ has to be a Veronese 3-manifold.
This completes the proof of Theorem \ref{main-thm}.


	\bigskip\bigskip

	\bibliographystyle{amsplain}

\end{document}